\documentclass[11pt,reqno]{amsart}
\usepackage{amssymb,amsmath}\usepackage{subfigure}\newcommand{\be}{\begin{eqnarray}}
\newcommand{\ee}{\end{eqnarray}}
\newcommand{\beno}{\begin{eqnarray*}}
\newcommand{\eeno}{\end{eqnarray*}}
\newcommand{\const}{\mbox{\rm const.}}\newcommand{\rp}{\frac{1}{p}}\newcommand{\e}{{\varepsilon}}\newcommand{\R}{{\mathbb R}}\newcommand{\Z}{{\mathbb Z}}\newcommand{\B}{{\mathcal B}}\newcommand{\D}{{\mathcal D}}\newcommand{\M}{{\mathcal M}}\newcommand{\Ik}{{\mathcal I}}

\newcommand{\om}{\omega}
\newtheorem{theorem}{Theorem}\newtheorem{lemma}[theorem]{Lemma}\newtheorem{assumption}[theorem]{Assumption}\newtheorem{cor}[theorem]{Corollary}\newtheorem{prop}[theorem]{Proposition}\theoremstyle{definition}\newtheorem{defi}[theorem]{Definition}\theoremstyle{remark}\newtheorem{remark}[theorem]{Remark}\numberwithin{equation}{section}\input epsf.sty
\begin{document}\thispagestyle{empty}

%
%
%
%
%
%
%BRAND NEW COMMANDS%
%
%
%
%
%
%
\newcommand{\G}{{\mathcal G}}
\newcommand{\C}{{\mathbb C}}
\newcommand{\Pa}{{P_{1,\theta}(y)}}
\newcommand{\Pb}{{P_{2,\theta}(y)}}
\newcommand{\Pshp}{{P_{1,t}^\sharp(x)}}
\newcommand{\Pflt}{{P_{1,t}^\flat(x)}}
\newcommand{\OTL}{{\Omega(\theta,\ell)}}
\newcommand{\rsz}{{R(\theta^*)}}
\newcommand{\phitil}{{\tilde{\varphi}_t}}
\newcommand{\lambdatil}{{\tilde{\lambda}}}
\newcommand{\al}{\alpha}
\newcommand{\ftpI}{\langle (g^2 + \tau^2 f^2)^{\frac{p}{2}} \rangle_{I}}
\newcommand{\fI}{\langle f \rangle_{I}}
\newcommand{\fpI}{\langle |f|^p \rangle_{I}}
\newcommand{\gI}{\langle g \rangle_{I}}
\newcommand{\mtp}{|(f,h_J)| = |(g,h_J)|}
\newcommand{\fJp}{\langle f \rangle_{J^+}}
\newcommand{\fJm}{\langle f \rangle_{J^-}}
\newcommand{\gJp}{\langle f \rangle_{J^+}}
\newcommand{\gJm}{\langle f \rangle_{J^-}}
\newcommand{\fIpm}{\langle f \rangle_{J^{\pm}}}
\newcommand{\gIpm}{\langle g \rangle_{J^{\pm}}}
\newcommand{\fIp}{\langle f \rangle_{I^{+}}}
\newcommand{\fIm}{\langle f \rangle_{I^{-}}}
\newcommand{\gIp}{\langle g \rangle_{I^{+}}}
\newcommand{\gIm}{\langle g \rangle_{I^{-}}}
\newcommand{\yone}{y_1}
\newcommand{\ytwo}{y_2}
\newcommand{\yonep}{y_1^+}
\newcommand{\ytwop}{y_2^+}
\newcommand{\yonem}{y_1^-}
\newcommand{\ytwom}{y_2^-}
\newcommand{\ythr}{y_3}

\newcommand{\xone}{x_1}
\newcommand{\xtwo}{x_2}
\newcommand{\xonep}{x_1^+}
\newcommand{\xtwop}{x_2^+}
\newcommand{\xonem}{x_1^-}
\newcommand{\xtwom}{x_2^-}
\newcommand{\xthr}{x_3}
\newcommand{\ptwo}{\frac{p}{2}}
\newcommand{\alp}{\al^+}
\newcommand{\alm}{\al^-}
\newcommand{\alpm}{\al^{\pm}}
\newcommand{\twrp}{\frac{2}{p}}
\newcommand{\sbeta}{\sqrt{\om^2 - \tau^2}}
\newcommand{\ssbeta}{\om^2 - \tau^2}
\newcommand{\sign}{\operatorname{sign}}
\newcommand{\rst}[1]{\ensuremath{{\mathbin\upharpoonright}%
\newcommand{\and}{\operatorname{and}}

\raise-.5ex\hbox{$#1$}}}

\title[Operator Norm of a Perturbation of the Martingale Transform]{{Perturbation of Burkholder's martingale transform and Monge--Amp\`ere equation}}
\author{Nicholas Boros}\address{Nicholas Boros, Dept. of Math., Michigan State University.
{\tt borosnic@msu.edu}}
\author{Prabhu Janakiraman}\address{Prabhu Janakiraman, Dept. of Math., Michigan State University.
\newline{\tt pjanakir1978@gmail.com}}
\author{Alexander Volberg}\address{Alexander Volberg, Dept. of  Math., Michigan State University. 
{\tt volberg@math.msu.edu}}%\,.

\begin{abstract}Let $\{d_k\}_{k \geq 0}$ be a complex martingale difference in $L^p[0,1],$ where $1<p<\infty,$ and $\{\e_k\}_{k \geq 0}$ a sequence in $\{\pm 1\}.$  We obtain the following generalization of Burkholder's famous result.  If $\tau \in [-\frac 12, \frac 12]$ and $n \in \Z_+$ then 
\beno
\left\|\sum_{k=0}^n{\left(\begin{array}{c}
\e_k\\
\tau \end{array}\right) d_k}\right\|_{L^p([0,1], \C^2)} \leq ((p^*-1)^2 + \tau^2)^{\frac 12}\left\|\sum_{k=0}^n{d_k}\right\|_{L^p([0,1], \C)}, 
\eeno
where $((p^*-1)^2 + \tau^2)^{\frac 12}$ is sharp and $p^*-1 = \max\{p-1, \frac 1{p-1}\}.$  For $2\leq p<\infty$ the result is also true with sharp constant for $\tau \in \R.$
\end{abstract}
\maketitle
\section{{\bf Introduction}} \label{sec:intro}
In a series of papers, \cite{B1} to \cite{Bu7}, Burkholder was able to compute the $L^p$ operator norm of the martingale transform, which we will denote as $MT.$  This was quite a revolutionary result, not only because of the result itself but because of the method for approaching the problem.  Burkholder's method in these early papers was inspiration for the Bellman function technique, which has been a very useful tool in approaching modern and classical problems in harmonic analysis (this paper will demonstrate the Bellman function technique as well).  But, the result itself has many applications.  One particular application of his result is for obtaining sharp estimates for singular integrals.   Consider the Ahlfors-Beurling operator, which we will denote as $T.$  Lehto, \cite{Le}, showed in 1965 that $\|T\|_p := \|T\|_{p \to p} \geq (p*-1) = \max\left\{p-1, \frac{1}{p-1}\right\}.$  Iwaniec conjectured in 1982, \cite{Iw},  that $\|T\|_p = p^*-1.$  The only progress toward proving that conjecture has been using Burkholder's result, see \cite{NV}, \cite{Ba1} and \cite{Ba} for the major results toward proving the conjecture.  However, Burkholder's estimates have been useful for lower bound estimates as well.  For example, Geiss, Montgomery-Smith and Saksman, \cite{GMS}, were able to show that $\|\Re T\|_p, \|\Im T\|_p \geq p^*-1,$ by using Burkholder's estimates.  The upper bound for these two operators were determined as $p^*-1$ by Nazarov, Volberg, \cite{NV} and Ba\~{n}uelos, M\`endez-Hern\`andez \cite{Ba1}, so we now have $\|\Re T\|_p = \|\Im T\|_p = p^*-1.$   Note that $\Re T$ the difference of the squares of the planar Riesz transforms, i.e. $T = R_1^2 - R_2^2.$

A recent result of Geiss, Montomery-Smith and Saksman, \cite{GMS} points to the following observation, though not immediately.  We can estimate linear combinations of squares of Riesz transforms if we know the corresponding estimate for a linear combination of the martingale transform and the identity operator.  In other words, one can get at estimates of the norm of $(R_1^2 - R_2^2) + \tau \cdot I,$ by knowing the estimates of the norm of $MT + \tau \cdot I.$  $\|MT + \tau \cdot I\|_p$ has only been computed for either $\tau =0$ by Burkholder \cite{Bu3} or $\tau = \pm 1$ by Choi \cite{Choi}.  The problem is still open for all other $\tau-$values and seems to be very difficult, though we have had some progress.  But, if we consider ``quadratic" rather than linear perturbations then things become more manageable (see \cite{BJV}, \cite{BJV1}).  This brings us to the focus of this paper, which is determining estimates for quadratic perturbations of the martingale transform, which will have connections to quadratic combinations of squares of Riesz transforms. 
 
To prove our main result we are going to take a slightly indirect approach.  Burkholder (see \cite{Bu3}) defined the martingale transform, $MT_{\e},$ as 
\beno
MT_{\e}\left(\sum_{k=1}^n{d_k}\right) := \sum_{k=1}^n\e_k{d_k}.  
\eeno
Then the main result can be stated as 
\beno 
\sup_{\vec{\e}}\left\|\left(\begin{array}{c}
MT_{\vec{\e}}\\
\tau I \end{array} \right)\right\|_{L^p(\C) \to L^{p}(\C^2)} = \sup_{\vec{\e}}\frac{\left\|\sum_{k=1}^n{\left(\begin{array}{c}
\e_k\\
\tau \end{array}\right) d_k}\right\|_{p}}{\|\sum_{k=1}^n{d_k}\|_p} = ((p^*-1)^2 + \tau^2)^{\frac 12},
\eeno
where $I$ is the identity transformation and $\tau$ is ``small".  However, rather than working with this martingale transform in terms of the martingale differences, in a probabilistic setting, we will define another martingale transform in terms of the Haar expansion of $L^p[0,1]$ functions and set up a Bellman function in that context.  Burkholder showed, in \cite{Bu7}, that these two different martingale transforms have the same $L^p$ operator norm, for $\tau = 0,$  so we expected a perturbation of these to act similarly and it turns out that they do.  For convenience, we will work with the martingale transform in the Haar setting.  Using the Bellman function technique will turn the problem of finding the sharp constant of the above estimate into solving a second order partial differential equation.  The beauty of this approach is that it gets right to the heart of the problem with very little advanced techniques needed in the process.  In fact, the only background material that is needed for the Bellman function technique approach, is some basic knowledge of partial differential equations and some elementary analysis. 

Observe that for $2 \leq p <\infty,$ the estimate from above is just an application of Minkowski's inequality on $L^{\frac p2}$ and Burkholder's original result.  But, this argument does not address sharpness, even though the constant obtained turns out to be the sharp constant for small $\tau$.  For $1<p<2,$ Minkowski's inequality (in $l^{\frac 2p}$) also plays a role, but to a lesser extent and cannot give the sharp constant, as we will see Proposition \ref{Bellmanlowerbound}.  It is, indeed, very strange that such sloppy estimation could give the estimate with sharp constant for $1 \leq p < \infty.$  We will now rigorously develop some background ideas needed to set up the Bellman function.

In our calculations we follow the scheme of \cite{VaVo3}, but our ``Dirichlet problem" for Monge--Amp\`ere is different.  For small $\tau$ the scheme works.  For large $\tau$ and $1<p<2$ it definitely must be changed as \cite{BJV} shows.  The amazing feature is the ``splitting" of the result to two quite different cases: $1<p<2$ and $2\leq p < \infty,$ where in the former case we know the result only for small $\tau,$ but in the latter one $\tau$ is unrestricted. 

\subsection{\bf Motivation of the Bellman function}\label{}
Let $I$ be an interval and $\al^{\pm} \in \R^+$ such that $\al^+ + \al^- = 1$.  These $\al^\pm$ generate two subintervals $I^\pm$ such that $|I^\pm| = \al^\pm |I|$ and $I = I^- \cup I^+$.  We can continue this decomposition indefinitely as follows.  Any sequence $\{\al_{n,m}:0<\al_{n,m}<1, 0 \leq m <2^n, 0 <n<\infty, \al_{n,2k} + \al_{n,2k+1} =1 \}$, generates the collection $\Ik := \{I_{n,m}:0 \leq m < 2^n, 0<n<\infty\}$ of subintervals of $I$, where $ I_{n,m} = I_{n,m}^- \cup I_{n,m}^+ = I_{n+1,2m+1} \cup I_{n+1,2m+1}$ and $ \al^- = \al_{n+1,2m}, \al^+ = \al_{n+1,2m+1}$.  Note that $I_{0,0} = I$.  

For any $J \in \Ik$ we define the Haar function $h_J := -\sqrt{\frac{\al^+}{\al^- |J|}}\chi_{J^-} + \sqrt{\frac{\al^-}{\al^+ |J|}}\chi_{J^+}$.  If $\max \{|I_{n,m}|:0\leq m<2^n\} \to 0$ as $n \to \infty$ then $\{h_{J}\}_{J \in \Ik}$ is an orthonormal basis for $L_0^2(I):= \{f \in L^2(I): \int_I{f} = 0\}.$  However, if we add one extra function then Haar functions form an orthonormal basis in $L^2[0,1].$   Fix $I_0 = [0,1]$ and $\Ik = \D$ as the dyadic subintervals of $I$.  Let $\D_n = \{I \in \D: |I| = 2^{-n}\}.$  We use the notation $\langle f \rangle_J$ to represent the average integral of $f$ over the interval $J \in \D$ and $\sigma(\D_n)$ to be the $\sigma$-algebra generated by $\D_n.$  For any $f \in L^1(I_0)$ we have the identity 
\be \label{haarexpansionoff}
\sum_{I\in D_n} {\langle f \rangle_I \chi_I} = \langle f \rangle_{(I_0)} \chi_{[0,1)} + \sum_{I \in \sigma(\D_n)}{(f,h_I)h_I}.
\ee
By Lebesgue differentiation, the left-hand side in (\ref{haarexpansionoff}) converges to $f$ almost everywhere, as $n \to \infty.$  So any $f \in L^p(I_0) \subset L^1(I_0)$ can be decomposed in terms of the Haar system as 
\beno
f = \langle f \rangle_{(I_0)} \chi_{(I_0)} + \sum_{I \in \D}{(f,h_I)h_I}.
\eeno  

In terms of the expansion in the Haar system we define the martingale transform, $g$ of $f,$ as 
\beno
g := \langle g \rangle_{(I_0)} \chi_{(I_0)} + \sum_{I \in \D}{\e_I(f,h_I)h_I},
\eeno  
where $\e_I \in \{\pm 1\}.$  Requiring that $|(g,h_J)| = |(f,h_J)|,$ for all $J \in \D,$ is equivalent to $g$ being the martingale transform of $f,$ for $f,g \in L^p(I_0).$   

Now we define the Bellman function as $\B(x_1,x_2,x_3):=$
\beno
\sup_{f,g}\{\ftpI: x_1 = \fI, x_2 = \gI,x_3 = \fpI, \mtp, \,\, \forall J \in \D\}
\eeno
on the domain $\Omega = \{x \in \R^3: x_3 \geq 0, |x_1|^p \leq x_3\}.$  The Bellman function is defined in this way, since we would like to know the value of the supremum of $\left\|\left(\begin{array}{c}
\ g\\
\ \tau f \end{array}\right)\right\|_p,$ where $g$ is the martingale transform of $f$.  Note that $|x_1|^p \leq x_3$ is just H\"{o}lder's inequality.  Even though the Bellman function is only being defined for real-valued functions, we can ``vectorize" it to work for complex-valued (and even Hilbert-valued) functions, as we will later demonstrate.  Finding the Bellman function will make proving the following main result quite easy.  We will call $\langle (g^2 + \tau^2 f^2)^{\frac p2} \rangle_I^{\frac 1p}$ the ``quadratic perturbation" of the martingale transform's norm $\langle |g|^p \rangle_I^{\frac 1p}.$

\begin{theorem}\label{mainresult}
Let $\{d_k\}_{k \geq 1}$ be a complex martingale difference in $L^p[0,1],$ where $1<p<\infty,$ and $\{\e_k\}_{k \geq 1}$ a sequence in $\{\pm 1\}.$  If $\tau \in [-\frac 12, \frac 12]$ and $n \in \Z_+$ then 
\beno
\left\|\sum_{k=1}^n{\left(\begin{array}{c}
\e_k\\
\tau \end{array}\right) d_k}\right\|_{L^p([0,1], \C^2)} \leq ((p^*-1)^2 + \tau^2)^{\frac 12}\left\|\sum_{k=1}^n{d_k}\right\|_{L^p([0,1], \C)}, 
\eeno
where $((p^*-1)^2 + \tau^2)^{\frac p2}$ is sharp.  The result is also true with sharp constant for $2 \leq p < \infty$ and $\tau \in \R.$
\end{theorem}
Note that when $\tau = 0$ we get Burkholder's famous result \cite{Bu3}.   

Now that we have the problem formalized, notice that $\B$ is independent of the initial choice of $I_0$ (which we will just denote $I$ from now on) and $\{\al_{n,m}\}_{n,m}$, so we return to having them arbitrary.  Finding $\B$ when $p=2$ is easy, so we will do this first.

\begin{prop}
If $p=2$ then $\B(x) = x_2^2 -x_1^2 + (1+ \tau^2)x_3.$
\end{prop}
\begin{proof}
Since $f \in L^2(I)$ then $f = \fI \chi_I + \sum_{J \in \D}{(f,h_J)h_J}$ implies
\beno
\langle |f|^2 \rangle_I & = & \frac{1}{|I|}\int_I{|f|^2}\\ 
& = &  \fI^2 + 2 \fI \sum_{J \in \D}{(f,h_J)}\frac{1}{|I|}\int_I{h_J} + \frac{1}{|I|}\int_I{\sum_{J,K \in \D}{(f,h_J)(f,h_K)h_J h_K}}\\
& = & \fI^2 + \frac{1}{|I|}\sum_{J \in \D}{|(f,h_J)|^2}.\\
\eeno
So $\|f\|_2^2 = |I|x_3 = |I| x_1^2 + \sum_{J \in \D}{|(f,h_J)|^2}$ and similarly 
\beno 
\|g\|_2^2 = |I| x_2^2 + \sum_{J \in \D}{|(g,h_J)|^2} = |I| x_2^2 + \sum_{J \in \D}{|(f,h_J)|^2}.
\eeno
Now we can compute $\B$ explicitly, $(p = 2)$
\beno
\ftpI & = & \langle |g|^2 \rangle_I + \tau^2 \langle |f|^2 \rangle_I
= x_2^2 + \tau^2 x_1^2 + (1 + \tau^2)\frac{1}{|I|}\sum_{J \in \D}{|(f,h_J)|^2}\\
& = & x_2^2 + \tau^2 x_1^2 + (1 + \tau^2)(x_3-x_1^2). \qedhere 
\eeno 
\end{proof}

\subsection{\bf Outline of Argument to Prove Main Result}

Computing the Bellman function, $\B,$ for $p \neq 2,$ is much more difficult, so more machinery is needed. In Section \ref{Bellmanfunctionproperties} we will derive properties of the Bellman function, the most notable of which is concavity under certain conditions.  Finding a $\B$ to satisfy the concavity will amount to solving a partial differential equation, after adding an assumption.  This PDE has a solution on characteristics that is well known, so we just need to find an explicit solution from this, using the Bellman function properties.  How the characteristics behave in the domain of definition for the Bellman function will give us several cases to consider.  In Section \ref{MongeAmpsolution} we will get a Bellman function candidate for $1<p<\infty$ by putting together several cases.  Once we have what we think is the Bellman function, we need to show that it has the necessary smoothness and that Assumption \ref{degenerateassumption} was not too restrictive to give us the Bellman function.  This is covered in Section \ref{showingwehavetrueBellmanfunction}.  Finally the main result is shown in Section \ref{sectionmainresult}.  In Section \ref{addendum}, we show why several cases did not lead to a Bellman function candidate and why the value of $\tau$ was restricted for the Bellman function candidate.
 
\subsection{\bf Properties of the Bellman function}\label{Bellmanfunctionproperties}
One of the properties we nearly always have (or impose) for any Bellman function, is concavity (or convexity).   It is not true that $\B$ is globally concave, on all of $\Omega,$ but under certain conditions it is concave.  The needed condition is that $g$ is the martingale transform of $f,$ or $|x_1^+ -x_1^-| = |x_2^+ -x_2^-|$ in terms of the variables in $\Omega.$ 

\begin{defi}
We say that the function $B$ on $\Omega$ has restrictive concavity if for all $x^{\pm} \in \Omega$ such that $x = \al^+ x^+ + \al^- x^-, \al^+ +\al^- =1$ and $|x_1^+ -x_1^-| = |x_2^+ -x_2^-|$ then $\B(x) \geq \al^+ \B(x^+) + \al^- \B(x^-).$
\end{defi}

\begin{prop}\label{concavity}
The Bellman function $\B$ is restrictively concave in the $x-$variables.  
\end{prop}

\begin{proof}
Let $\e > 0$ be given and $x^{\pm} \in \Omega$.  By the definition of $\B$, there exists $f^{\pm}, g^{\pm}$ on $I^{\pm}$ such that $\fIpm = x_1^{\pm},  \gIpm = x_2^{\pm}, \langle|f^{\pm}|^p\rangle_{I^{\pm}} = x_3^{\pm}$ and 
\beno
\B(x^{\pm}) - \langle[(g^{\pm})^2+\tau^2 (f^{\pm})^2]^{\frac{p}{2}}\rangle_{I^{\pm}} \leq \e
\eeno
On $I = I^+ \cup I^-$ we define $f$ and $g$ as $f := f^+ \chi_{I^+}+ f^- \chi_{I^-}, g := g^+ \chi_{I^+} + g^- \chi_{I^-}.$  So, 
\beno |x_1^+ - x_1^-| & = & |\fIp - \fIm| = \left|\frac{1}{|I^+|}\int_{|I^-|}{f}-\frac{1}{|I^-|}\int_{I^-}{f}\right|\\
& = & \left|\frac{1}{\al^+|I|}\int_{|I^-|}{f}-\frac{1}{\al^-|I|}\int_{I^-}{f}\right| = \frac{1}{|I|}\left|\int{f\left(\frac{1}{\al^+}\chi_{I^+}-\frac{1}{\al^-}\chi_{I^-}\right)}\right|\\
& = & \sqrt{\frac{|I|}{\al^+ \al^-}}\left|\int{f h_I}\right| =: \sqrt{\frac{|I|}{\al^+ \al^-}}\left|(f,h_I)\right|.\\
\eeno
Similarly,$|x_2^+ - x_2^-| = \sqrt{\frac{|I|}{\al^+ \al^-}}\left|(g,h_I)\right|.$  So our assumption $|x_1^+ - x_1^-| = |x_2^+ - x_2^-|$ is equivalent to $|(f, h_I)| = |(g, h_I)|.$  Since $x_1 = \fI, x_2 = \gI$ and $x_3 = \fpI$ then $f$ and $g$ are test functions and so 
\beno
\B(x) & \geq & \ftpI\\ 
&=& \al^+ \langle[(g^{+})^2+\tau^2 (f^{+})^2]^{\frac{p}{2}}\rangle_{I^{+}} + \al^- \langle[(g^{-})^2+\tau^2 (f^{-})^2]^{\frac{p}{2}}\rangle_{I^{-}}\\
&\geq& \al^+ \B(x^+) + \al^- \B(x^-) - \e. \qedhere\\
\eeno
\end{proof}

At this point we don not quite have concavity of $\B$ on $\Omega$ since there is the restriction $|x_1^+ - x_1^-| = |x_2^+ - x_2^-|$ needed.  To make this condition more manageable, we will make a change of coordinates.  Let $ y_1 :=  \frac{x_2 + x_1}{2}, y_2 :=  \frac{x_2 - x_1}{2}$ and $y_3 := x_3$.  We will also change notation for the Bellman function and corresponding domain in the new variable $y$.  Let $\M(y_1, y_2, y_3) := \B(x_1,x_2,x_3) = \B(y_1 - y_2,y_1 + y_2,y_3).$  Then the domain of definition for $\M$ will be $\Xi := \{y \in \R^3: y_3 \geq 0, |y_1 - y_2|^p \leq y_3\}.$  

If we consider $x^{\pm} \in \Omega$ such that $|x_1^+ - x_1^-| = |x_2^+ - x_2^-|,$ then the corresponding points $y^{\pm} \in \Xi$ satisfy either $\yonep = \yonem$ or $\ytwop = \ytwom.$  This implies that fixing $\yone$ as $\yonep = \yonem$ or $\ytwo$ as $\ytwop = \ytwom$ will make $\M$ concave with respect to $\ytwo, \ythr$ under fixed $\yone$ and with respect to $\yone, \ythr$ under $\ytwo$ fixed.  

Rather than using Proposition \ref{concavity} to check the concavity of the Bellman function we can just check it in the following way, assuming $\M$ is $C^2.$  Let $j \neq i \in \{1,2\}$ and fix $y_i$ as $y_i^+ = y_i^-.$  Then $\M$ as a function of $y_j, y_3$ is concave if \[\left( \begin{array}{cc}
\M_{y_jy_j} & \M_{y_jy_3}\\
\M_{y_3y_j} & \M_{y_3y_3} \end{array} \right) \leq 0,\] which is equivalent to 
\beno
\M_{y_j y_j} \leq 0, \M_{y_3 y_3} \leq 0, D_j = \M_{y_j y_j}\M_{y_3 y_3} - \M_{y_3 y_j} \M_{y_j y_3} \geq 0.
\eeno

\begin{prop}\label{D_i}
(Restrictive Concavity in $y-$variables)  Let $j \neq i \in \{1,2\}$ and fix $y_i$ as $y_i^+ = y_i^-.$  If $\M_{y_j y_j} \leq 0, \M_{y_3 y_3} \leq 0$ and $D_j = \M_{y_j y_j}\M_{y_3 y_3} - (\M_{y_j y_3})^2 \geq 0$ for $j = 1$ and $j = 2$ then $\M$ is Restrictively concave.
\end{prop}
The Bellman function, as it turns out, has many other nice properties.

\begin{prop} \label{Bellmanprop}
Suppose that $\M$ is $C^1(\R^3),$ then $\M$ has the following properties.

(i)  Symmetry:  $\M(\yone, \ytwo, \ythr) = \M(\ytwo,\yone,\ythr) = \M(-\yone, -\ytwo, \ythr)$

(ii)  Dirichlet boundary data:  $\M(\yone, \ytwo, (\yone - \ytwo)^{p}) = ((\yone + \ytwo)^2 + \tau^2 (\yone - \ytwo)^2)^{\frac{p}{2}}$

(iii)  Neumann conditions:  $\M_{\yone}=\M_{\ytwo}$ on $\yone = \ytwo$ and $\M_{\yone}=-\M_{\ytwo}$ on $\yone = -\ytwo$

(iv) Homogeneity:  $\M(r\yone, r\ytwo, r^p \ythr) = r^p\M(\yone, \ytwo, \ythr), \forall r > 0$

(v)  Homogeneity relation:  $\yone \M_{\yone} + \ytwo \M_{\ytwo} + p \ythr \M_{\ythr} = p\M$
\end{prop}

\begin{proof}
(i)  Note that we get $\B(\xone, \xtwo, \xthr) = \B(-\xone, \xtwo, \xthr) = \B(\xone, -\xtwo, \xthr)$ by considering test functions $\widetilde{f} = -f$ and $\widetilde{g} = -g$.  Change coordinates from $x$ to $y$ and the result follows.

(ii)  On the boundary $\{\xthr = |\xone|^p\}$ of $\Omega$ we see that 
\beno
\frac{1}{|I|} \int_I{|f|}^p = \fpI = \xthr = |\xone|^p = |\fI|^p = \left|\frac{1}{|I|} \int_I{f}\right|^p
\eeno
is only possible if $f \equiv \const$  (i.e. $f = \xone$).  But, $\mtp$ for all $J \in \Ik$, which implies that $g \equiv \const$ (i.e. $g = \xtwo$).  Then $\B(\xone, \xtwo, |\xone|^p) = \ftpI = (\xtwo^2 + \tau^2 \xone^2)^\ptwo.$  Changing coordinates gives the result.

(iii)  This follows from from (i).

(iv)  Consider the test functions $\widetilde{f} = rf, \widetilde{g} = rg$

(v)  Differentiate (iv) with respect to $r$ and evaluate it at $r=1.$ \qedhere
\end{proof}

Now that we have all of the properties of the Bellman function we will turn our attention to actually finding it.  Proposition \ref{D_i} gives us two partial differential inequalities to solve, $D_1 \geq 0, D_2 \geq 0$, that the Bellman function must satisfy.  Since the Bellman function is the supremum of the left-hand side of our estimate under the condition that $g$ is the martingale transform of $f,$ and must also satisfy the estimates in Proposition \ref{D_i}, then it seems reasonable that the Bellman function (being the optimal such function) may satisfy the following, for either $j= 1$ or $j=2$:
\beno
D_j = \M_{y_jy_j}\M_{y_3y_3} - (\M_{y_3y_j})^2 = 0.
\eeno
The PDE that we now have is the well known Monge--Amp\`ere equation which has a solution.  Let us make it clear that we have added an assumption.

\begin{assumption}\label{degenerateassumption}
$D_j = \M_{y_jy_j}\M_{y_3y_3} - (\M_{y_3y_j})^2 = 0,$ for either $j=1$ or $j=2.$
\end{assumption}
Adding this assumption comes with a price.  Any function that we construct, satisfying all properties of the Bellman function, must somehow be shown to be the Bellman function.  We will refer to any function satisfying some, or all Bellman function properties as a Bellman function candidate.  
\begin{prop} \label{Pogorelov}
For $j=1$ or $2,$ $\M_{y_jy_j}\M_{y_3y_3} - (\M_{y_3y_j})^2 = 0$ has the solution $M(y) = y_jt_j + y_3t_3+t_0$ on the characteristics $y_jdt_j + y_3dt_3 + dt_0=0$, which are straight lines in the $y_j \times y_3$ plane.  Furthermore, $t_0, t_j, t_3$ are constant on characteristics with the property $M_{y_j} = t_j, M_{y_3} = t_3.$ 
\end{prop}
This is a result of Pogorelov, see \cite{Pog}, \cite{VaVo}.  Now that we have a solution $M$ to the Monge--Amp\`ere, we need get rid of $t_0, t_j, t_3$ so that we have an explicit form of $M,$ without the characteristics.  We note that a solution to the Monge--Amp\`ere is not necessarily the Bellman function.  It must satisfy the restrictive concavity of Proposition \ref{D_i}, be $C^1$-smooth, and satisfy the properties of Proposition \ref{Bellmanprop}.  The restrictive concavity property is one of the key deciding factors of whether or not we have a Bellman function in many cases.  Even if the Monge--Amp\`ere solution satisfies all of those conditions, it must still be shown to be equal to the Bellman function, because we added an additional assumption (Assumption \ref{degenerateassumption}) to get the Monge--Amp\`ere solution as a starting point. This will be considered rigorously in Section \ref{showingwehavetrueBellmanfunction}, after we obtain a solution to the Monge--Amp\`ere equation, with the appropriate Bellman function properties.  So from this point on we will use $M$ and $B$ to denote solutions to the Monge--Amp\`ere equation, i.e. Bellman function candidates, and $\M$ and $\B$ to denote the true Bellman function. 

\section{\bf Computing the Bellman function candidate from the Monge--Amp\`ere solution}\label{MongeAmpsolution}
Due to the symmetry property of $\M,$ from Proposition \ref{Bellmanprop}, we only need to consider a portion of the domain $\Xi,$ which we will denote as, $\Xi_+ := \{y: -\yone \leq \ytwo \leq \yone, \ythr \geq 0, (\yone -\ytwo)^p \leq \ythr \}.$  Since the characteristics are straight lines, then one end of each line must be on the boundary $\{y: (\yone - \ytwo)^p = \ythr\}$.  Let $U$ denote the point at which the characteristic touches the boundary.  Furthermore, the characteristics can only behave in one of the following four ways, since they are straight lines in the plane:

(1)  The characteristic goes from $U$ to $\{y:\yone = \ytwo\}$

(2)  The characteristic goes from $U$ to to infinity, running parallel to the $\ythr$-axis

(3)  The characteristic goes from $U$ to $\{y:\yone = -\ytwo\}$ 

(4)  The characteristic goes from $U$ to $\{y:(\yone -\ytwo)^p = \ythr \}$

To find a Bellman function candidate we must first fix a variable ($\yone$ or $\ytwo$) and a case for the characteristics.  Then we use the Bellman function properties to get rid of the characteristics.  If the Monge--Amp\`ere solution satisfies restrictive concavity, then it is a Bellman function candidate.  However, checking the restrictive concavity is quite difficult in many of the cases, since it amounts to doing second derivative estimates for an implicitly defined function.  Let us now find our Bellman function candidate.

\begin{remark}\label{casesandnotation}
Since we will have either $\yone$ or $\ytwo$ fixed in each case, then there will be eight cases in all.  Let $(1_j), (2_j), (3_j), (4_j)$ denote the case when $M_{y_jy_j}M_{y_3y_3} - (M_{y_3y_j})^2 = 0$ and $y_i$ is fixed, where $i \neq j$.  Also, we will denote $G(z_1, z_2):= (z_1 +z_2)^{p-1}[z_1 - (p-1)z_2]$ and $\om := \left(\frac{\M(y)}{y_3}\right)^{\frac 1p}$ from this point on.
\end{remark}

\subsection{\bf Bellman candidate for $2<p<\infty$}\label{solutionpgreaterthan2}

The solution to the Monge--Amp\`ere equation when $2<p<\infty,$ is only partially valid on the domain in two cases, due to restrictive concavity.  Case $(1_2),$ will give us an implicit solution that is valid on part of $\Xi_+$ and Case $(2_2)$ will give us an explicit solution for the remaining part of $\Xi_+.$  First, we deal with Case $(1_2).$

\subsubsection{\bf Case $(1_2)$}\label{partial1_2subsec}
Since we are considering Case $(1_2),$ then $\yone \geq 0$ is fixed until the point that we have the implicit solution independent of the characteristics satisfying all of the Bellman function properties.

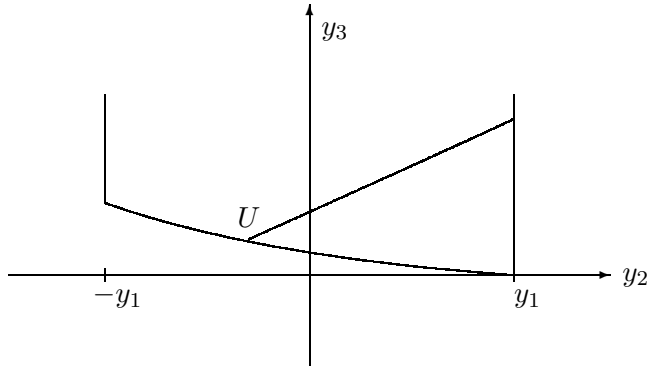
\begin{figure}[h]

\setlength{\unitlength}{0.8cm}

\begin{picture}(10,6)(-5,-1)
\put(-5,0){\vector(1,0){10}} %axis
\put(0,-1.5){\vector(0,1){6}} %axis
\put(5.2,-0.1){$y_2$} %label

\put(-3.6,-.45){$-y_1$} %label
\put(-3.4,-.1){\line(0,1){.2}} %tick

\put(-1.2,.8){$U$} %label

\put(3.4, -.45){$y_1$} %label
\put(0.2,4){$y_3$} %label

\qbezier(3.4,0)(-0.8853,0.3)(-3.4,1.2) %parabola

\put(-3.4,1.2){\line(0,1){1.8}} %first vertical line
\qbezier(-1,.6)(3,2.4)(3.4,2.6)
\put(3.4,-.1){\line(0,1){3.1}} %last vertical line

\end{picture}

\caption{Sample characteristic of solution from Case $(1_2)$} 
\label{characteritic1_2}

\end{figure}

\begin{prop}\label{solution1_2}
For $1<p<\infty$ and $\frac{p-2}{p}\yone < \ytwo < \yone$, $M$ is given implicitly by the relation $G(\yone +\ytwo, \yone -\ytwo) = \ythr G(\sqrt{\om^2 -\tau^2}, 1),$ where $G(z_1, z_2) := (z_1+z_2)^{p-1}[z_1-(p-1)z_2]$ on $z_1 + z_2 \geq 0$ and $\omega := \left(\frac{M(y)}{\ythr}\right)^{\frac 1p}.$
\end{prop}

This is proven through a series of Lemmas.

\begin{lemma}\label{reductionsolution1_2}
$M(y) = t_2 y_2 + t_3 y_3 + t_0$ on the characteristic $y_2 dt_2 + y_3 dt_3 + dt_0 = 0$ can be simplified to $M(y) = \left(\frac{\sqrt{(\yone + u)^2 + \tau^2 (\yone -u)^2}}{\yone -u}\right)^p\ythr,$
where $u$ is the unique solution to the equation $\frac{\ytwo + (\frac{2}{p} -1)\yone}{\ythr} = \frac{u + (\frac{2}{p} -1)\yone}{(\yone -u)^p}$ and $\frac{p-2}{p} \yone < \ytwo < \yone$
\end{lemma}

\begin{proof}
A characteristic in Case $(1_2)$ is from $U = (\yone, u, (\yone - u)^p)$ to $W = (\yone, \yone, w).$  Throughout the proof we will use the properties of the Bellman function from Proposition \ref{Bellmanprop}.  Using the Neumann property and the property from Proposition \ref{Pogorelov} we get $M_{y_1} = M_{y_2} = t_2$ at $W.$  By homogeneity at $W$ we get 
\beno
py_2 t_2 + pwt_3 + pt_0 = pM(W) = y_1 M_{y_1} + \ytwo M_{y_2} + p\ythr M_{y_3} = 2\yone t_2 + pwt_3.
\eeno
Then $t_0 = (\frac{2}{p}-1)\yone t_2$ and $dt_0 = (\frac{2}{p}-1)\yone dt_2$, since $y_1$ is fixed.  So $M(y) = [\ytwo + (\twrp-1)\yone]t_2 + y_3t_3$ on $[\ytwo + (\twrp-1)\yone]dt_2 + y_3dt_3 = 0$.  By substitution we get, $M(y) = y_3[t_3 -t_2 \frac{dt_3}{dt_2}]$ on characteristics.  But, $t_2, t_3, \frac{dt_3}{dt_2}$ are constant on characteristics, which gives that $\frac{M(y)}{y_3} \equiv \const$ as well.  We can calculate the value of the constant by using the Dirichlet boundary data for $M$ at $U.$  Therefore, $M(y) = \left(\frac{\sqrt{(\yone + u)^2 + \tau^2 (\yone -u)^2}}{\yone -u}\right)^p\ythr,$
where $u$ is the solution to the equation 

\be
\label{chreq}
\frac{\ytwo + (\frac{2}{p} -1)\yone}{\ythr} = \frac{u + (\frac{2}{p} -1)\yone}{(\yone -u)^p}.
\ee

Now fix $u = -(\twrp -1)\yone$.  Then we see that $\ytwo = -(\twrp -1)\yone = u$ is also fixed by (\ref{chreq}).  This means that the characteristics are limited to part of the domain, as shown in ~Figure ~\ref{Sectorfor1_2}, since they start at $U$ and end at $W \in \{y_1 = y_2\}$.
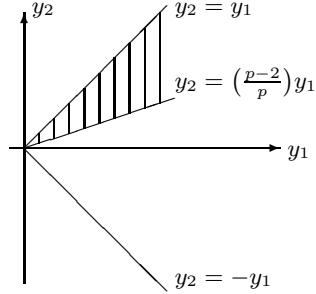
\begin{figure}[h]
%\begin{center}
\setlength{\unitlength}{0.02cm}
\begin{picture}(200,200)(0,0)
\thinlines
\put(20,10){\vector(0,1){180}} %x axis
\put(10,100){\vector(1,0){180}} %y axis
\put(20,100){\line(1,1){95}} %y2 = y1
\put(20,100){\line(1,-1){95}} %y2 = -y1
\put(20,100){\line(3,1){100}} %y2 = (p-2)/p y1
\put(30,110){\line(0,-1){6.6}} %1st vertical line
\put(40,120){\line(0,-1){13.3}} %2nd vertical line
\put(50,130){\line(0,-1){20}}
\put(60,140){\line(0,-1){26.7}}
\put(70,150){\line(0,-1){33.3}}
\put(80,160){\line(0,-1){40}}
\put(90,170){\line(0,-1){46.7}}
\put(100,180){\line(0,-1){53.3}}
\put(110,190){\line(0,-1){60}}
\put(120,190){\footnotesize $y_2=y_1$}
\put(120,10){\footnotesize $y_2=-y_1$}
\put(120,140){\footnotesize $y_2=\big(\frac{p-2}{p}\big)y_1$}
\put(150,80){\footnotesize $$}
\put(25,190){\footnotesize $y_2$}
\put(195,95){\footnotesize $y_1$}
\end{picture}
\caption{Sector for characteristics in Case $(1_2),$ when $p>2.$}
\label{Sectorfor1_2}
%\end{center}
\end{figure}
All that remains is verifying the equation (\ref{chreq}) has exactly one solution $u = u(\yone, \ytwo,\ythr)$ in the sector $\frac{p-2}{p} \yone < \ytwo < \yone$. Indeed, the function
\begin{equation*}
f(u):=y_3\left[u+\left(\frac2p-1\right)y_1\right]-(y_1-u)^p\left[y_2+\left(\frac2p-1\right)y_1\right]
\end{equation*}
is monotone increasing for $u<y_1$, 
$f(-(\frac2p -1)y_1) = -\big(\frac2p y_1)^p\Big[y_2+\big(\frac2p-1\big)y_1\Big] < 0$ and 
$f(y_1) = \frac2p y_1y_3 > 0.$  Therefore, we do get a unique solution, $u,$ in the sector.\qedhere
\end{proof}

\begin{lemma}
$M(y) = \left(\frac{\sqrt{(\yone + u)^2 + \tau^2 (\yone -u)^2}}{\yone -u}\right)^p\ythr$  can be rewritten as $G(\yone +\ytwo, \yone -\ytwo) = \ythr G(\sqrt{\om^2 -\tau^2}, 1)$ for $\frac{p-2}{p} \yone < \ytwo < \yone.$ 
\end{lemma}
\begin{proof}
$\om = \left(\frac{M(y)}{\ythr}\right)^{\frac1p} = \frac{\sqrt{(\yone + u)^2 + \tau^2 (\yone -u)^2}}{\yone -u} \geq |\tau|\frac{|\yone - u|}{\yone - u} = |\tau|$

Since $ \yone \pm u \geq 0$ and $\om^2 - \tau^2 \geq 0$, then $u = \frac{\sqrt{\om^2 - \tau^2}-1}{\sqrt{\om^2 - \tau^2}+1}\yone$ by inversion.  Substituting this into $\frac{\ytwo + (\frac{2}{p} -1)\yone}{\ythr} = \frac{u + (\frac{2}{p} -1)\yone}{(\yone -u)^p}$ gives
\beno
2^{p-1}\yone^{p-1}[p\ytwo-(p-2)\yone] = \ythr(\sqrt{\om^2-\tau^2}+1)^{p-1}[\sqrt{\om^2-\tau^2}-(p-1)]
\eeno
or $(x_1 + x_2)^{p-1}[x_2-(p-1)x_1] = \left[\sqrt{B^{\frac2p} - \left(\tau x_3^{\frac1p}\right)^2}+x_3^\rp\right]^{p-1}\left[\sqrt{B^{\frac2p} - \left(\tau x_3^{\frac1p}\right)^2}-(p-1)x_3^\rp\right].$ Thus, $G(x_2, x_1) = G\left(\sqrt{B^{\frac2p} - \left(\tau x_3^{\frac1p}\right)^2}, x_3^\rp\right)$ or $G(\yone + \ytwo, \yone - \ytwo) = \ythr G(\sqrt{\om^2 - \tau^2}, 1).\qedhere$

\end{proof}
This proves Proposition \ref{solution1_2}.  We have constructed a partial Bellman function candidate from the Monge--Amp\`ere solution in Case$(1_2),$ so $\yone$ no longer needs to be fixed.  All of the properties of the Bellman function were used to derive this partial Bellman candidate, but the restrictive concavity from Proposition \ref{D_i} still needs to be verified.  To verify restrictive concavity, we need that $M_{y_2 y_2} \leq 0, M_{y_3 y_3} \leq 0, D_2 \geq 0$ and $M_{y_1 y_1} \leq 0, D_1 \geq 0.$  By assumption $D_2 = 0,$ so we needn't worry about that estimate.  The remaining estimates will be verified in a series of Lemmas.  The first Lemma is an idea taken from Burkholder \cite{Burkholder} to make the calculations for computing mixed partials shorter.  In the Lemma, we compute the partials of arbitrary functions which we will choose specifically later, although it is not hard to see what the appropriate choices should be.

\begin{lemma} \label{partialsofM}
Let $H = H(y_1, y_2), \Phi(\om) = \frac{H(y_1, y_2)}{y_3}, R_1=R_1(\omega):=\frac1{\Phi'}$ and  $R_2=R_2(\omega):= R'_1=-\frac{\Phi''}{{\Phi'}^2}.$  Then 

$M_{y_{_3}y_{_3}} = \frac{p\omega^{p-2}R_1H^2}{y_3^3}[\omega R_2+(p-1)R_1]$

$M_{y_{_3}y_{_i}} = -\frac{p\omega^{p-2}R_1HH'}{y_3^2}[\omega R_2+(p-1)R_1]$

$M_{y_{_i}y_{_i}} = \frac{p\omega^{p-2}R_1}{y_3^{\phantom3}}\left([\omega R_2+(p-1)R_1](H')^2+\omega y_3H''\right)$ 

$D_i=M_{y_{_3}y_{_3}}M_{y_{_i}y_{_i}}-M_{y_{_3}y_{_i}}^2=
\frac{p^2\omega^{2p-3}R_1^2H^2H''}{y_3^3}[\omega R_2+(p-1)R_1].$
\end{lemma} 

\begin{proof}
First of all we calculate the partial derivatives of $\omega$:
\begin{align*}
\Phi'\omega_{y_{_3}}=-\frac H{y_3^2}\quad
&\implies\quad\omega_{y_{_3}}=-\frac{R_1H}{y_3^2}\,,
\\
\Phi'\omega_{y_{_i}}=\frac{H_{y_{_i}}}{y_3}\quad
&\implies\quad\omega_{y_{_i}}=\frac{R_1H_{y_{_i}}}{y_3}=\frac{R_1H'}{y_3},\qquad
i=1,2\,.
\end{align*}
Here and further we shall use notation $H'$ for any partial derivative
$H_{y_i}$, $i=1,2$. This cannot cause any confusion since only one $i \in \{1,2\}$
participate in calculation of $D_i$. 
\begin{align*}
\omega_{y_{_3}y_{_3}}&=-\frac{R_2\omega_{y_{_3}}H}{y_3^2}+2\frac{R_1H}{y_3^3}
=\frac{R_1H}{y_3^4}(R_2H+2y_3)\,,
\\
\omega_{y_{_3}y_{_i}}&=-\frac{R_2\omega_{y_{_i}}H}{y_3^2}-\frac{R_1H'}{y_3^2}
=-\frac{R_1H'}{y_3^3}(R_2H+y_3)\,,
\\
\omega_{y_{_i}y_{_i}}&=\frac{R_2\omega_{y_{_i}}H}{y_3}+\frac{R_1H'}{y_3}
=\frac{R_1}{y_3^2}(R_2(H')^2+y_3H'')\,.
\end{align*}

Now we pass to the calculation of derivatives of $M=y_3\omega^p$:
\begin{align*}
M_{y_{_3}}&=py_3\omega^{p-1}\omega_{y_{_3}}+\omega^p\,,
\\
M_{y_{_i}}&=py_3\omega^{p-1}\omega_{y_{_i}}\,;
\end{align*}
\begin{align}
M_{y_{_3}y_{_3}}&=py_3\omega^{p-1}\omega_{y_{_3}y_{_3}}+
2p\omega^{p-1}\omega_{y_{_3}}+p(p-1)y_3\omega^{p-2}\omega_{y_{_3}}^2
\notag
\\
\label{M33}
&=\frac{p\omega^{p-2}R_1H^2}{y_3^3}[\omega R_2+(p-1)R_1]\,,
\\
\notag
M_{y_{_3}y_{_i}}&=py_3\omega^{p-1}\omega_{y_{_3}y_{_i}}+
p\omega^{p-1}\omega_{y_{_i}}+p(p-1)y_3\omega^{p-2}\omega_{y_{_3}}\omega_{y_{_i}}
\\
\notag
&=-\frac{p\omega^{p-2}R_1HH'}{y_3^2}[\omega R_2+(p-1)R_1]\,,
\\
\notag
M_{y_{_i}y_{_i}}&=py_3\omega^{p-1}\omega_{y_{_i}y_{_i}}
+p(p-1)y_3\omega^{p-2}\omega_{y_{_i}}^2
\\
\label{Mii}
&=\frac{p\omega^{p-2}R_1}{y_3^{\phantom3}}\left([\omega
R_2+(p-1)R_1](H')^2+\omega y_3H''\right)\,.
\end{align}
This yields
\begin{equation}
\label{Di} D_i=M_{y_{_3}y_{_3}}M_{y_{_i}y_{_i}}-M_{y_{_3}y_{_i}}^2=
\frac{p^2\omega^{2p-3}R_1^2H^2H''}{y_3^3}[\omega R_2+(p-1)R_1]\,.\qedhere
\end{equation}
\end{proof}
 
\begin{lemma}\label{H''1_2}
If $\al_i, \beta_i \in \{\pm 1\}$ and $H(\yone, \ytwo) = G(\al_1 \yone +\al_2 \ytwo, \beta_1 \yone +\beta_2 \ytwo)$ then  \begin{displaymath}
   H'' = \left\{
     \begin{array}{lr}
       4G_{z_1 z_2}, \al_j = \beta_j\\
       0, \al_j = -\beta_j.
     \end{array}
   \right.
\end{displaymath} 
Consequently, in Case $(1_2),$ $\sign H'' = -\sign (p-2).$
\end{lemma} 

\begin{proof}
\begin{align*}
H''&=\frac{\partial^2}{\partial y_i^2}
G(\alpha_1y_1+\alpha_2y_2,\beta_1y_1+\beta_2y_2)
\\
&=\alpha_i^2G_{\!z_{_1}\!z_{_1}}\!\!+2\alpha_i\beta_iG_{\!z_{_1}\!z_{_2}}
\!\!+\beta_i^2G_{\!z_{_2}\!z_{_2}}
\\
&=G_{\!z_{_1}\!z_{_1}}\!\!+G_{\!z_{_2}\!z_{_2}}\!\!\pm2G_{\!z_{_1}\!z_{_2}}\,,
\end{align*}
where the ``$+$'' sign has to be taken if the coefficients in front of $y_i$
are equal and the ``$-$'' sign in the opposite case.

The derivatives of $G$ are simple:
\begin{align*}
G_{\!z_1}\!&=p(z_1+z_2)^{p-2}\big[z_1-(p-2)z_2\big]\,,
\\
G_{\!z_2}\!&=-p(p-1)z_2(z_1+z_2)^{p-2}\,;
\end{align*}
\begin{align*}
G_{\!z_{_1}\!z_{_2}}\!&=p(p-1)(z_1+z_2)^{p-3}\big[z_1-(p-3)z_2\big]\,,
\\
G_{\!z_{_1}\!z_{_2}}\!&=-p(p-1)(p-2)z_2(z_1+z_2)^{p-3}\,,
\\
G_{\!z_{_2}\!z_{_2}}\!&=-p(p-1)(z_1+z_2)^{p-3}\big[z_1+(p-1)z_2\big]\,.
\end{align*}

Note that
$G_{\!z_{_1}\!z_{_1}}\!\!+G_{\!z_{_2}\!z_{_2}}\!\!=2G_{\!z_{_1}\!z_{_2}}$, and
therefore, \begin{displaymath}
   H'' = \left\{
     \begin{array}{lr}
       4G_{z_1 z_2}, \al_j = \beta_j\\
       0, \al_j = -\beta_j
     \end{array}
   \right.
\end{displaymath} 

Now in Case $(1_2)$, we must choose $\alpha_1 = 1, \alpha_2 = 1, \beta_1 = 1$ and $\beta_2 = -1$ for $H$ to match how the implicit solution was defined in terms of $G$ in Proposition \ref{solution1_2}.  Then $G_{\!z_{_1}\!z_{_2}}=-p(p-1)(p-2)(y_1-y_2)(2y_1)^{p-3}.$ \qedhere
\end{proof}

\begin{remark} \label{lowerbndbeta}
Let $\beta := \sbeta$ from this point on.  In Case $(1_2), \beta > p-1$ in the sector $\frac{p-2}{p}\yone < \ytwo < \yone.$  Equivalently, $B > (\tau^2+(p-1)^2)^{\frac p2}$ in $\frac{p-2}{p}\yone < \ytwo < \yone.$
\end{remark}
This is an easy application of Proposition \ref{solution1_2}:
\beno
(\beta + 1)^{p-1}[\beta - p+1] = G(\beta, 1) & = & \frac{1}{\ythr}G(\yone+\ytwo, \yone-\ytwo)\\ 
& = & (2y_1)^{p-1}[-(p-2)y_1 + py_2]>0.\\
\eeno
Before we can compute the signs of $M_{\yone \yone}, M_{\ytwo \ytwo}, M_{\ythr \ythr}$ and $D_1$ we need a technical Lemma.
\begin{lemma}\label{gtaup}
If $1 < p < \infty$ and $\tau \in \R,$ then
\beno
g(\beta) := - p(p-2) \om \beta^{-3} (\beta + 1)^{p-3}[(\tau^2 + p-1)\beta^2 - \tau^2(p-3)\beta + \tau^2]
\eeno
satisfies $\sign g(\beta) = - \sign (p-2)$ in Case $(1_2).$
\end{lemma}

\begin{proof}
The only terms controlling the sign in $g$ are $(p-2)$ and the quadratic part, which we will denote $q(\beta)$.  So we need to simply figure out the sign of $q.$  The discriminant of $q$ is $\tau^2(p-1)[\tau^2(p-5)-4].$  

If $p \leq 5$ then the discriminant of $q$ is negative and so $q(\beta)>0.$  If $p>5,$ and $\tau^2(p-5)-4<0$ then $q(\beta)>0$ once again.  

The only case left to consider is if $p>5$ and $\tau^2(p-5)-4 \geq 0.$  The zeros of $q$ are given by $\beta = \frac{\tau^2(p-3) \pm |\tau| \sqrt{p-1} \sqrt{\tau^2(p-5) -4}}{2(\tau^2+p-1)}.$  Let $\beta_1, \beta_2$ be the zeros such that $\beta_2 \geq \beta_1.$ We claim that $\max\{p-1, \beta_2\} = p-1.$  Indeed, $p-1 - \beta_2 > 0$ 
\beno
& \Longleftrightarrow & (p+1)\tau^2 + 2(p-1)^2 > |\tau| \sqrt{p-1} \sqrt{\tau^2(p-5) -4}\\
& \Longleftrightarrow & 4(p-1)^4+4\tau^2(p+1)(p-1)^2+\tau^4(p+1)^2>\tau^2(p-1)(\tau^2(p-5)-4)\\
& \Longleftrightarrow & (p-1)^4+\tau^2p^2(p-1)+\tau^4(2p-1)>0,
\eeno
which is obviously true for all $\tau \in \R.$  Now that we have proven the claim, recall that $\beta > p-1,$ as shown in Remark \ref{lowerbndbeta}.  Therefore, $\beta > \beta_2,$ so $q(\beta) > 0$ in this case.  \qedhere
\end{proof}    

\begin{lemma}\label{D1whenpgreaterthan2}
$D_1 > 0$ in Case $(1_2)$ for all $\tau \in \R.$
\end{lemma}

\begin{proof}
We use the partial derivatives of $G$ computed in the proof of Lemma \ref{H''1_2} to make the computations of $\Phi'$ and $\Phi''$ easier.
\be 
\Phi(\om) & = & G(\beta, 1)\notag \\ \label{Phiofomega}
\Phi'(\om) & = & p \om [\beta + 1]^{p-2}[1-(p-2)\beta^{-1}] \\ \notag 
\Phi''(\om) & = & p (\beta + 1)^{p-2}[1-(p-2)\beta^{-1}] + p(p-2)\om^2 \beta^{-1}[\beta + 1]^{p-3}[1-(p-2)\beta^{-1}]\\\notag
& \,\,& \,\,\,\,\,\,\,\,\,\, + p(p-2)\om^2 \beta^{-3}[\beta - 1]^{p-2}\\\notag
\Lambda & = & (p-1)\Phi' - \om \Phi''\\\notag
& = & p (p-2) \om \beta^{-1}(\beta + 1)^{p-2}[1-(p-2)\beta^{-1}]\\ \notag
& \,\,& \,\,\,\,\,\,\,\,\,\, -  p (p-2) \om^3 \beta^{-3} (\beta + 1)^{p-3}[\beta(\beta -p+2)+\beta+1]\\\notag
& = & p (p-2) \om \beta^{-2} (\beta + 1)^{p-3}[\beta - p +2]\{\beta(\beta + 1) - \om^2\} - p(p-2)\om^3 \beta^{-3} (\beta + 1)^{p-2}\\\notag
& = & p(p-2) \om \beta^{-3}[\beta(\beta -p+2)(\beta-\tau^2)-\om^2(\beta + 1)]\\\notag
& = & - p(p-2) \om \beta^{-3} (\beta + 1)^{p-3}[(\tau^2 + p-1)\beta^2 - \tau^2(p-3)\beta + \tau^2] \label{Biglambda}\\ 
\ee
So we can see that $\sign \Lambda = \sign g(\beta) = -\sign (p-2),$ by Lemma \ref{gtaup}.  Therefore, $\sign D_1 = \sign H'' \sign \Lambda = [-\sign(p-2)]^2$ by (\ref{Di}) and Lemma \ref{H''1_2}. \qedhere
\end{proof}

Since $D_1 > 0,$ then all that remains to be checked, for the restrictive concavity of $M,$ is that $M_{y_i y_i}$ (for $i = 1,2$) and $M_{y_3 y_3}$ have the appropriate signs.  But, it turns out that only for $2<p<\infty,$ will these have the appropriate signs.

\begin{lemma}\label{My3y3pgreaterthan2}
$\sign M_{\yone \yone} = \sign M_{\ytwo \ytwo} = \sign M_{\ythr \ythr} = -\sign (p-2)$ in Case $(1_2)$ for all $\tau \in \R.$  Therefore, $M$ is a partial Bellman function candidate for $2< p <\infty$ but not for $1<p<2,$ since it does not satisfy the required restrictive concavity.
\end{lemma}

\begin{proof}
By (\ref{M33}),
\beno
M_{\ythr \ythr} = \frac{p \om^{p-2}R_1^2H^2}{y_3^3}\left[\frac{\Lambda}{\Phi'}\right]\\
\eeno
Remark \ref{lowerbndbeta} gives $\Phi' > 0.$  From Lemma \ref{gtaup}, $\sign M_{\ythr \ythr} = \sign \Lambda = \sign g(\beta) = -\sign (p-2).$  By (\ref{Mii}), for $i=1$or $2,$
\beno
M_{y_i y_i} &=& \frac{p \omega^{p-2}R_1}{\ythr}\left[(\omega R_2 + (p-1)R_1)(H')^2+\omega \ythr H''\right]\\
&=& \frac{p \omega^{p-2}}{\ythr (\Phi')^3}\left[\Lambda(H')^2+\omega \ythr H''(\Phi')^2\right],\\
\eeno
giving $\sign M_{\ytwo \ytwo} = -\sign (p-2).$  \qedhere
\end{proof}

The previous two lemmas established that the partial Bellman function candidate, from Case $(1_2)$ satisfies the restrictive concavity property, for $2<p<\infty.$  The candidate was constructed using the remaining Bellman function properties, so it is in fact a partial candidate.  Now we will turn our attention to Case $(2)$.  As it turns out, Case $(2_2)$ also gives a partial Bellman function candidate, which, as luck would have it, is the missing half of the parital Bellman candidate just constructed.  

\subsubsection{\bf Case $(2_2)$ for $2<p<\infty$}
We can obtain a Bellman candidate from Case $(2)$ without having to separately fix $\yone$ or $\ytwo.$  Let us compute the solution in this case.  

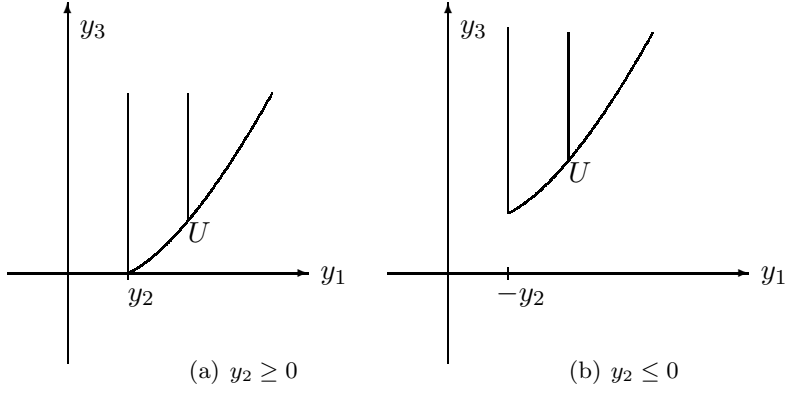
\begin{figure}[h]
  \centering
 \setlength{\unitlength}{0.8cm}
\subfigure[$y_2 \geq 0$]{
\begin{picture}(6,6)(0,-1)
\put(-1,0){\vector(1,0){5}} %axis
\put(0,-1.5){\vector(0,1){6}} %axis
\put(4.2,-0.1){$y_1$} %label
\put(2,.52){$U$} %label
\put(1, -.45){$y_2$} %label
\put(0.2,4){$y_3$} %label

\qbezier(1,0)(2,.5)(3.4,3) %parabola

\put(2,.9){\line(0,1){2.1}} %characteristic
\put(1,-.1){\line(0,1){3.1}} %last vertical line
\end{picture}}                
\subfigure[$y_2 \leq 0$]
{\setlength{\unitlength}{0.8cm}
\begin{picture}(6,6)(0,-1)
\put(-1,0){\vector(1,0){6}} %axis
\put(0,-1.5){\vector(0,1){6}} %axis
\put(5.2,-0.1){$y_1$} %label
\put(2,1.52){$U$} %label
\put(0.8, -.45){$-y_2$} %label
\put(1,-.1){\line(0,1){.2}} %tick
\put(0.2,4){$y_3$} %label

\qbezier(1,1)(2,1.5)(3.4,4) %parabola

\put(2,1.9){\line(0,1){2.1}} %characteristic
\put(1,1){\line(0,1){3.1}} %last vertical line

\end{picture}}
\caption{Sample characteristic of Monge--Amp\`ere solution in Case $(2_1)$}
\label{characteristic2_1}
\end{figure}

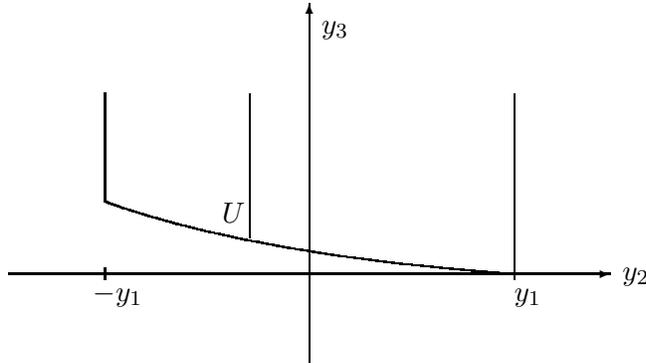
\begin{figure}[h]

\setlength{\unitlength}{0.8cm}

\begin{picture}(10,6)(-5,-1)
\put(-5,0){\vector(1,0){10}} %axis
\put(0,-1.5){\vector(0,1){6}} %axis
\put(5.2,-0.1){$y_2$} %label

\put(-3.6,-.45){$-y_1$} %label
\put(-3.4,-.1){\line(0,1){.2}} %tick

\put(-1.45,.85){$U$} %label

\put(3.4, -.45){$y_1$} %label
\put(0.2,4){$y_3$} %label

\qbezier(3.4,0)(-0.8853,0.3)(-3.4,1.2) %parabola

\put(-3.4,1.2){\line(0,1){1.8}} %first vertical line

\put(-1,.6){\line(0,1){2.4}}
\put(3.4,-.1){\line(0,1){3.1}} %last vertical line

\end{picture}
\caption{Sample characteristic of Monge--Amp\`ere solution for Case $(2_2)$} 
\label{characteritic2_2}
\end{figure}

\begin{lemma}\label{case2solution}
In Case $(2)$ we obtain 
\be \label{case2sol}
M(y) = (1+\tau^2)^{\frac p2}[y_1^2 + 2 \gamma y_1 y_2 + y_2^2]^{\frac p2} + c [y_3 - (y_1 - y_2)^p]
\ee
as a Bellman function candidate, where $c>0$ is some constant and $\gamma = \frac{1-\tau^2}{1+\tau^2}.$
\end{lemma}

\begin{proof}
In Case $(2)$, on the characteristic $y_i dt_i + y_3dt_3 + dt_0 = 0$, $\yone$ and $\ytwo$ are fixed.  Furthermore, on the characteristic, $t_0, t_i, t_3$ are fixed, so we have
\beno
M(y) &=& y_i t_i + \ythr t_3 + t_0\\
&=& (y_i t_i + t_0) + y_3 t_3\\
&=& c_1(y_1,y_2) + c_2(y_1, y_2) y_3\\
\eeno
Then $M_{\ythr \ythr} = 0$ and $M_{\ythr y_i} = \partial_{y_i}c_2.$  Recall that $D_i \geq 0$ by Remark \ref{D_i}, so $\partial_{y_i}c_2(y_1, y_2) = 0.$  This implies that $c_2$ is a constant.  Using the boundary data from Proposition \ref{Bellmanprop} gives $((\yone + \ytwo)^2+\tau^2(\yone - \ytwo)^2)^{\frac p2} = M(\yone,\ytwo, (\yone - \ytwo)^p) = c_1(\yone, \ytwo) + c_2(\yone - \ytwo)^p.$  Solving for $c_1(\yone, \ytwo)$ gives the result.  To see that $c_2>0,$ just notice that as $\ythr \to \infty, M(y) \to \infty.\qedhere$
\end{proof}

It is not possible to determine if this Bellman function candidate satisfies restrictive concavity, unless we know the value of the constant $c$ in Lemma \ref{case2solution}.  This constant can be computed by using the fact that (\ref{case2sol}) must agree with the partial candidate in Case $(1_2)$ at $y_2 = \frac{p-2}{p}\yone$, if (\ref{case2sol}) is in fact a candidate itself. 

\begin{lemma}
In Case $(2_2),$ the value of the constant in Lemma \ref{case2solution} is $c = ((p-1)^2 + \tau^2)^{\frac p2}$ for $2<p<\infty.$
\end{lemma}

\begin{proof}
If $M(y) = (1+\tau^2)^{\frac p2}[y_1^2 + 2 \gamma y_1 y_2 + y_2^2]^{\frac p2} + c [y_3 - (y_1 - y_2)^p]$ (where $\gamma = \frac{1-\tau^2}{1+\tau^2}$) is to be a candidate, or partial candidate, then it must agree at $y_2 = \frac{p-2}{p}\yone,$ with the solution $M$ given implicitly by the relation $G(\yone +\ytwo, \yone -\ytwo) = \ythr G(\sqrt{\om^2 -\tau^2}, 1)$, from Proposition \ref{solution1_2}.  At $y_2 = \frac{p-2}{p}\yone,$ 
\beno
(\sbeta + 1)^{p-1}[\sbeta - p+ 1] &=& G(\sbeta, 1)\\ 
&=& \frac{1}{\ythr}(2y_1)^{p-1}[-(p-2)y_1 + (p-2)y_1] = 0.\\  
\eeno
Since $\sbeta + 1 \neq 0$ then $\sbeta = p-1,$ which implies $\om = ((p-1)^2 + \tau^2)^{\frac 12}.$  So,  
\beno
((p-1)^2 + \tau^2)^{\frac p2} \ythr &=& \om^p \ythr\\ 
&=& M(\yone, \frac{p-2}{p}\yone, \ythr)\\
&=& \left[\left(2\frac{p-1}{p}y_1\right)^2 + \tau^2\left(\frac 2p \yone\right)^2\right]^{\frac p2} + c\left[\ythr - \left(\frac p2 \yone\right)^p\right].\\
\eeno
Now just solve for $c.$ \qedhere
\end{proof}

\subsubsection{\bf Gluing together partial candidates from Cases $(1_2)$ and $(2_2)$}\label{gluingp>2}

It turns out that the Bellman function candidate obtained from Case $(2_2)$ is only valid on part of the domain $\Xi_+$, since it does not remain concave throughout (for example at or near $(\yone, \yone, \ythr)$).  As luck would have it, the partial candidate has the necessary restrictive concavity on the part of the domain where the candidate from Case $(1_2)$ left off, i.e. in $-\yone < \ytwo < \frac{p-2}{p}\yone.$  This means that we can glue together the partial candidate from Cases $(1_2)$ and $(2_2)$ to get a candidate on $\Xi_+$ for $2<p<\infty.$  The characteristics for this solution can be seen in ~Figure ~\ref{gluedchar1_2and2_2}.

\begin{figure}[h]

\setlength{\unitlength}{0.8cm}

\begin{picture}(10,6)(-5,-1)
\put(-5,0){\vector(1,0){10}} %axis
\put(0,-1.5){\vector(0,1){6}} %axis
\put(5.2,-0.1){$y_2$} %label

\put(-3.6,-.45){$-y_1$} %label
\put(-3.4,-.1){\line(0,1){.2}} %tick

\put(1.5,-.6){$\frac{p-2}{p}y_1$} %label
\put(1.4,-.1){\line(0,1){.2}}%tick

\put(3.4, -.45){$y_1$} %label
\put(0.2,4){$y_3$} %label

\qbezier(3.4,0)(-0.8853,0.3)(-3.4,1.2) %parabola

\put(-3.4,1.2){\line(0,1){1.8}} %vertical lines
\put(-2.8,.97){\line(0,1){2.03}}
\put(-2.2,.8){\line(0,1){2.2}}
\put(-1.6,.7){\line(0,1){2.3}}
\put(-1,.6){\line(0,1){2.4}}
\put(-.4,.5){\line(0,1){2.5}}
\put(.2,.4){\line(0,1){2.6}}
\put(.8,.3){\line(0,1){2.7}}
\put(1.4,.2){\line(0,1){2.8}}
\put(1.45,.184){\line(1,6){.47}} %fan portion
\qbezier(1.5,.177)(2.524,2.2)(2.816,3)
\qbezier(1.55,.17)(3,2.1)(3.4,2.6)
\qbezier(1.72,.15)(3,1.12)(3.4,1.5)
\qbezier(1.8,.14)(3,.52)(3.4,.7)
\qbezier(2.15,.1)(3,.22)(3.4,.3)
\put(3.4,-.1){\line(0,1){3.1}} %last vertical line

\end{picture}

\caption{Characteristics of Bellman candidate for $2<p<\infty$} 
\label{gluedchar1_2and2_2}

\end{figure}
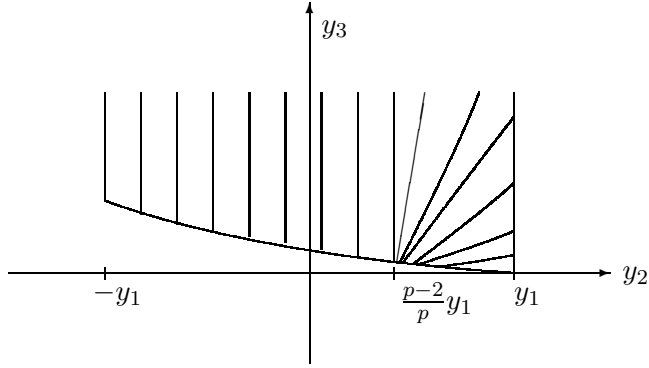

\begin{prop}\label{solution1_2and2_2glue}
For $2<p<\infty, \gamma = \frac{1-\tau^2}{1+\tau^2}$ and $\tau \in \R,$ the solution to the Monge--Amp\`ere equation is given by 
\beno
M(y) = (1+\tau^2)^{\frac p2}[y_1^2 + 2 \gamma y_1 y_2 + y_2^2]^{\frac p2} + ((p-1)^2 + \tau^2)^{\frac p2} [y_3 - (y_1 - y_2)^p] 
\eeno
when $-\yone < \ytwo \leq \frac{p-2}{p}\yone$ and is given implicitly by 

$G(\yone +\ytwo, \yone -\ytwo) = \ythr G(\sqrt{\om^2 -\tau^2}, 1)$ when $\frac{p-2}{p}\yone \leq \ytwo < \yone,$ where $G(z_1, z_2) = (z_1 + z_2)^{p-1}[z_1 - (p-1)z_2]$ and $\omega = \left(\frac{M(y)}{y_3} \right)^{\frac 1p}$  This solution satisfies all properties of the Bellman function.
\end{prop}

We already know that the implicit part of the solution has the correct restrictive concavity property of the Bellman function, as shown in Section \ref{partial1_2subsec}.  However, the restrictive concavity still needs to be verified for the explicit part.  Since the explict part of the solution satisfies $M_{\ythr y_i} = M_{\ythr \ythr} = 0,$ then $D_i = 0,$ for $i = 1,2.$  So all that remains to be verified for the restrictive concavity of the explicit part is checking the sign of $M_{y_i y_i},$ for $i = 1,2.$  Observe that the explicit part can be written as
\be \label{alternateexplicitform}
M(y) = [(\yone +\ytwo)^2+\tau^2(\yone -\ytwo)^2]^{\frac p2} + C_{p,\tau}[\ythr - (\yone -\ytwo)^p].
\ee 
It is easy to check that $M_{\ytwo \ytwo} \leq M_{\yone \yone}$ on $-\yone < \ytwo \leq \frac{p-2}{p}\yone$ for $2<p<\infty.$  So we only need to find the largest range of $\tau$'s such that $M_{\yone \yone} \leq 0.$  

\begin{lemma} \label{2_2glue3}
In Case $(2_2), M_{y_1 y_1} \leq 0$ on $-\yone < \ytwo \leq \frac{p-2}{p}\yone$ for all $\tau \in \R.$
\end{lemma}
\begin{proof}
Changing coordinates back to $x$ will make the estimates much easier.  So we would like to show that, on $0 \leq x_2 \leq (p-1)x_1,$ we have,
\be \label{plargertwoconcavityestimate}
M_{\yone \yone} \leq 0,
\ee
where $C_{p,\tau} = ((p-1)^2+\tau^2)^{\frac p2}$ and $\frac 1p M_{\yone \yone} =$ 
\beno
(p-2)(x_2^2+\tau^2 x_1^2)^{\frac{p-4}{2}}(x_2+\tau^2 x_1)^2 +(1+\tau^2)(x_2^2 +\tau^2 x_1^2)^{\frac{p-2}{2}}-(p-1)C_{p,\tau}x_1^{p-2}.
\eeno

First, consider $4\leq p < \infty.$  If $p \neq 4,$ then showing (\ref{plargertwoconcavityestimate}) is equivalent to 
\beno
(p-2)(p-1+\tau^2)^2+(1+\tau^2)((p-1)^2+\tau^2)-(p-1)((p-1)^2+\tau^2)^2 \leq 0,
\eeno
which can be verified using direct calculations, for all $\tau.$  Let $s = \frac{x_2}{x_1},$ then (\ref{plargertwoconcavityestimate}) simplifies to showing,
\beno
F(s) = (p-2)(s+\tau^2)^2+(1+\tau^2)(s^2+\tau^2) -C_{p, \tau}(p-1)(s^2+\tau^2)^{\frac{4-p}{2}} \leq 0,
\eeno
where $0 \leq s \leq p-1.$  For $p = 4, F$ is a quadratic function that is increasing on  $(\frac{-2\tau^2}{\tau^2+3}, p-1).$  Since $F(3) \leq 0,$ then $F(s) \leq 0$ on $(0,3).$

Now we will consider $2\leq p<4.$  Note that $F(s) = 0$ at $p=2,$ so we can assume that $p \neq 2.$  Breaking up the domain of $F$ will make things easier.  For $s \in (1, p-1),$ we have the following estimate, $(s+\tau^2)^2 \leq (s^2+\tau^2)^2.$  Let $t = s^2 + \tau^2,$ then 
\beno
\frac 1t F(s) \leq (p-2)t +1 +\tau^2 - C_{p,\tau}(p-1)t^{\frac{2-p}{2}}:=g_1(t).
\eeno   
Observe that $g_1$ is increasing on $1+\tau^2 \leq t \leq (p-1)^2+\tau^2$ and $g_1((p-1)^2+\tau^2) \leq 0.$ Therefore, $F(s) \leq 0$ on $(1, p-1).$  

For $s \in (0, \frac{1-\tau^2}{2}),$ we have the estimate $(s+\tau^2)^2 \leq s^2+\tau^2.$  Let $t = s^2 + \tau^2,$ then 
\beno
\frac 1t F(s) \leq p-1+\tau^2 - C_{p,\tau}(p-1)t^{\frac{2-p}{2}}:=g_2(t).
\eeno   
Since $g_2$ is increasing on $(\tau^2, \frac{(1-\tau)^2}{4}+\tau^2)$ and $g_2(\frac{(1-\tau)^2}{4}+\tau^2) \leq 0,$ then $F(s) \leq 0$ on $(1, p-1)$ and on $(0, \frac{1-\tau^2}{2}).$

All that remains is so show that $F(s) \leq 0$ on $(\frac{1-\tau^2}{2}, 1).$  If we estimate in the crudest possible way, on this interval, then we obtain:
\beno
\frac{1}{p-1}F(s) \leq (1+\tau^2)^2-((p-1)^2+\tau^2)^{\frac p2}\left(\frac{(1-\tau^2)^2}{4}+\tau^2\right)^{\frac{4-p}{2}} \leq 0, 
\eeno   
for all $|\tau| \leq 1$ and $3 \leq p \leq 4,$ by direct calculations.  So we need to estimate a little more carefully.  On $(\frac{1-\tau^2}{2}, 1),$ let $t = s^2 + \tau^2.$  So $t$ must be in the range $\frac{(1-\tau)^2}{4}-\tau^2 \leq t \leq 1+\tau^2.$  Then,
\beno
F(s) &=& (p-2)(\sqrt{t-\tau^2}+\tau^2)^2+(1+\tau^2)t-C_{p,\tau}(p-1)t^{\frac{4-p}{2}}\\
&=& (p-2)(t-\tau^2+2\tau^2\sqrt{t-\tau^2}+\tau^4)+(1+\tau^2)t-C_{p,\tau}(p-1)t^{\frac{4-p}{2}}\\
& \leq & (p-2)(t+\tau^2+\tau^4)+(1+\tau^2)t-C_{p,\tau}(p-1)t^{\frac{4-p}{2}} =:g_3(t)\\
\eeno 
One can see that $g_3$ is decreasing for $2 \leq p < 3.95$ and $g_3\left(\frac{(1-\tau^2)^2}{4}+\tau^2\right)\leq 0,$ for all $|\tau|\leq 1,$ by direct calculations.   Thus, $F(s) \leq 0$ on $(\frac{1-\tau^2}{2}, 1).$\qedhere  
\end{proof}

We have now verified that the explicit part of the Bellman function candidate, from Case $(2_2),$ has the appropriate restrictive concavity.  So we have proven Proposition \ref{solution1_2and2_2glue}, by Lemmas \ref{D1whenpgreaterthan2}, \ref{My3y3pgreaterthan2} and \ref{2_2glue3}.  Now that we have a Bellman candidate for $2<p<\infty,$ we will turn our attention to $p$-values in the dual range $1<p<2.$

\subsection{\bf The Bellman function candidate for $1<p<2$}\label{solutionplessthan2}

In order to get a Bellman function candidate for $1<p<2$ we just need to glue together candidates from Cases $(2_2)$ and $(3_2)$ in almost the same way as we did for $2<p<\infty$ in  Section \ref{solutionpgreaterthan2}.  Refer to Addendum 1 (Section \ref{addendum1}) for full details.  

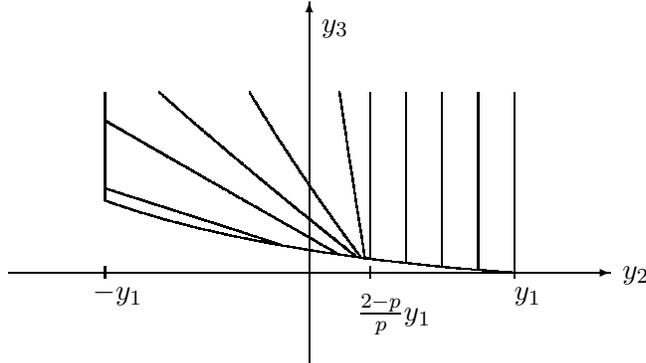
\begin{figure}[h]

\setlength{\unitlength}{0.8cm}

\begin{picture}(10,6)(-5,-1)
\put(-5,0){\vector(1,0){10}} %axis
\put(0,-1.5){\vector(0,1){6}} %axis
\put(5.2,-0.1){$y_2$} %label

\put(-3.6,-.45){$-y_1$} %label
\put(-3.4,-.1){\line(0,1){.2}} %tick

\put(.8,-.8){$\frac{2-p}{p}y_1$} %label
\put(1,-.1){\line(0,1){.2}}%tick

\put(3.4, -.45){$y_1$} %label
\put(0.2,4){$y_3$} %label

\qbezier(3.4,0)(-0.8853,0.3)(-3.4,1.2) %parabola

\put(-3.4,1.2){\line(0,1){1.8}} %vertical lines
%\put(-2.8,.97){\line(0,1){2.03}}
%\put(-2.2,.8){\line(0,1){2.2}}
%\put(-1.6,.7){\line(0,1){2.3}}
%\put(-1,.6){\line(0,1){2.4}}
%\put(-.4,.5){\line(0,1){2.5}}
%\put(.2,.4){\line(0,1){2.6}}
%\put(.8,.3){\line(0,1){2.7}}
%\put(1.4,.2){\line(0,1){2.8}}
%\put(2,.15){\line(0,1){2.85}}
%\put(2.6,.1){\line(0,1){2.9}}
%\put(3,.05){\line(0,1){2.95}}

%\put(1.45,.184){\line(1,6){.47}} %fan portion
%\qbezier(1.5,.177)(2.524,2.2)(2.816,3)
%\qbezier(1.55,.17)(3,2.1)(3.4,2.6)
%\qbezier(1.72,.15)(3,1.12)(3.4,1.5)
%\qbezier(1.8,.14)(3,.52)(3.4,.7)
%\qbezier(2.15,.1)(3,.22)(3.4,.3)
\qbezier(.92,.24)(.61,2.2)(.5,3) %fan portion
\qbezier(.84,.27)(-.22,1.7)(-1,3)
\qbezier(.76,.28)(-1,1.67)(-2.5,3)
\qbezier(.44,.33)(-2,1.72)(-3.4,2.53)
\qbezier(-.5,.47)(-2,.96)(-3.4,1.4)

\put(1,.25){\line(0,1){2.75}} % vertical lines
\put(1.6,.17){\line(0,1){2.83}} 
\put(2.2,.1){\line(0,1){2.9}} 
\put(2.8,.05){\line(0,1){2.95}} 
\put(3.4,-.1){\line(0,1){3.1}} %last vertical line

\end{picture}

\caption{Characteristics of Bellman candidate for $1<p<2$ and $|\tau| \leq \frac 12.$} 
\label{gluedchar2_2and3_2}

\end{figure}

\begin{prop}\label{solutionsfrom2_2and3_2}
Let $1<p<2$ and $\gamma = \frac{1-\tau^2}{1+\tau^2}.$  If $|\tau| \leq \frac 12,$ then a solution to the Monge--Amp\`ere equation is given by 

$M(y) = (1+\tau^2)^{\frac p2}[y_1^2 + 2 \gamma y_1 y_2 + y_2^2]^{\frac p2} + \left(\frac{1}{(p-1)^2} + \tau^2\right)^{\frac p2} [y_3 - (y_1 - y_2)^p]$ when $\frac{2-p}{p}\yone \leq \ytwo < \yone$ and is given implicitly by 

$G(\yone -\ytwo, \yone + \ytwo) = \ythr G(1, \sqrt{\om^2 -\tau^2})$ when $-\yone < \ytwo \leq \frac{2-p}{p}\yone,$ where $\omega = \left(\frac{M(y)}{y_3} \right)^{\frac 1p}.$  This solution satisfies all of the properties of the Bellman function.
\end{prop}

Most of the remaining cases do not yield a Bellman function candidate.  If we fix $\ytwo$ then the Monge--Amp\`ere solution from Cases $(1)$ and $(3)$ do not satisfy the restrictive concavity needed to be a Bellman function candidate.  Case $(2)$ yields the same partial solution if we first fix $\yone$ or $\ytwo,$ since restrictive concavity is only valid on part of the domain.  So, all that remains is Case $(4).$  However, we do not know whether or not Case $(4)$ gives a Bellman function candidate.  For $\tau = 0,$ it was shown in \cite{VaVo3} that Case $(4)$ does not produce a Bellman function candidate, since some simple extremal functions give a contradiction to linearity of the Monge--Amp\`ere solution on characteristics.  However, for $\tau \neq 0$ these extremal functions only work as a counterexample for some $p-$values and some signs of the martingale transform.  Case $(4)$ could give a solution throughout $\Xi_+$ or could yield a partial solution that would work well with the characteristics from Case $(2_1).$  Since Case $(4)$ does not provide a Bellman candidate for $\tau =0$, we expect the same for small $\tau.$  The picture probably changes most drastically for large $\tau.$  But it does not matter, since we will now show that our Bellman candidate is actually the Bellman function (which we would have to check anyways because of the added assumption).  The details for the remaining cases that do not yield a Bellman function candidate are in Addendum 2 (Section \ref{addendum}).

\section{\bf The Monge--Amp\`ere solution is the Bellman function}\label{showingwehavetrueBellmanfunction}
We will now show that the Monge--Amp\`ere solution obtained in Proposition \ref{solution1_2and2_2glue} and \ref{solutionsfrom2_2and3_2} is actually the Bellman function.  To this end, let us revert back to the $x-$variables.  We will denote the Bellman function candidate as $B_\tau$ and use $\B_\tau$ to denote the true Bellman function.  Extending the function $G$ on part of $\Omega_+$ to $U_\tau$ on all of $\Omega,$ appropriately, makes it possible to define the solution in terms of a single relation.  

\begin{defi}
Let  $v(x,y) := v_{p, \tau}(x, y) = (\tau^2|x|^2+|y|^2)^{\frac p2}-((p^*-1)^2+\tau^2)^{\frac p2}|x|^p,$ $u(x,y) := u_{p, \tau}(x, y) = p(1-\frac{1}{p*})^{p-1}\left(1+\frac{\tau^2}{(p^*-1)^2}\right)^{\frac{p-2}{2}}(|x|+|y|)^{p-1}[|y|-(p^*-1)|x|]$ and 
\begin{displaymath}
  U(x,y) :=  U_{p,\tau}(x, y) = \left\{
     \begin{array}{lr}
       v(x, y) & : |y| \geq (p^*-1)|x|\\
       u(x, y) & : |y| \leq (p^*-1)|x|.
     \end{array}
   \right.
\end{displaymath}
for $1<p<2.$  For $2<p<\infty$ we interchange the two pieces in $U.$ 
\end{defi}
\begin{prop}\label{unifyingsolution}
For $1<p<2$ and $|\tau|\leq \frac 12$ or $2<p<\infty$ and $\tau \in \R$ the Bellman function candidate is the unique positive solution given by 

\beno
U(x_1, x_2) = U\left(x_3^{\frac 1p}, \sqrt{B_\tau^{\frac 2p}-\tau^2x_3^{\frac 2p}}\right).
\eeno 
Furthermore, $U$ is $C^1-$smooth on $\Omega.$
\end{prop}

\begin{proof}
First consider $2 \leq p<\infty.$   It is clear that 
\be \label{unifyingrelation}
U(x_1, x_2) = U\left(x_3^{\frac 1p}, \sqrt{B_\tau^{\frac 2p}-\tau^2x_3^{\frac 2p}}\right),
\ee
by comparing the solution obtained in Proposition \ref{solution1_2and2_2glue} and using the symmetry property in Proposition \ref{Bellmanprop}.  The constant $\alpha_{p,\tau} = p(1-\frac{1}{p^*})^{p-1}\left(1+\frac{\tau^2}{(p^*-1)^2}\right)^{\frac{p-2}{2}}$ was determined so that $U_{x} = U_{y}$ at $|y| = (p^*-1)|x|.$  The partial derivatives are given by, 
\beno
u_x &=& \alpha_{p, \tau}(p-1)x'(|x|+|y|)^{p-2}(|y|-(p^*-1)|x|)-\alpha_{p, \tau}(p^*-1)x'(|x|+|y|)^{p-1},\\
v_x &=& p\tau^2 x(\tau^2|x|^2+|y|^2)^{\frac{p-2}{2}}-px'((p^*-1)^2+\tau^2)^{\frac p2}|x|^{p-1},\\
u_y &=& \alpha_{p,\tau}(p-1)y'(|x|+|y|)^{p-2}(|y|-(p^*-1)|x|)+\alpha_{p,\tau}y'(|x|+|y|)^{p-1},\\
v_y &=& py(\tau^2|x|^2+|y|^2)^{\frac{p-2}{2}},\\
\eeno
where $x' = \frac{x}{|x|}$ and $y' = \frac{y}{|y|}.$  $U$ is $C^1-$smooth, except possibly at gluing and symmetry lines.  It is easy to verify that $u_x$ is continuous at $\{x=0\},$ $U_x$ and $U_y$ are continuous at $\{|y| = (p^*-1)|x|\}$ and $v_y$ is continuous at $\{y=0\}.$  This proves that $U$ is $C^1-$smooth on $\Omega.$  

Observe that $U_{y}>0$ for $y\neq 0$ and $U_{x}<0$ for $x \neq 0.$  This is enough to show that $B_\tau$ is the unique positive solution to (\ref{unifyingrelation}).  Indeed, if $x\in \Omega$ such that $|x_1| = x_3^{\frac 1p},$ then $\sqrt{B_\tau^{\frac 2p}-\tau^2x_3^{\frac 2p}} = |x_2|$ by the Dirichlet boundary conditions.  This gives us (\ref{unifyingrelation}) uniquely at $B_\tau(x).$  Fix $x_1,$ such that $|x_1| < x_3^{\frac 1p},$ then $U\left(x_3^{\frac 1p}, \sqrt{B_\tau^{\frac 2p}-\tau^2x_3^{\frac 2p}}\right) < U\left(x_1, \sqrt{B_\tau^{\frac 2p}-\tau^2x_3^{\frac 2p}}\right).$  Since $x_1$ is fixed, then $\sqrt{B_\tau^{\frac 2p}-\tau^2x_3^{\frac 2p}} >|x_2|,$ so $U\left(x_1, \sqrt{B_\tau^{\frac 2p}-\tau^2x_3^{\frac 2p}}\right)$ strictly decreases to $U(x_1, x_2),$ as $\sqrt{B_\tau^{\frac 2p}-\tau^2x_3^{\frac 2p}}$ decreases to $|x_2|,$ giving us a unique $B_\tau(x)$ for which (\ref{unifyingrelation}) holds. 

Now consider $1<p<2.$  $U$ is $C^1-$smooth on $\Omega,$ since $v_x$ is continuous at $\{x=0\},$ $u_y$ is continuous at $\{y=0\}$ and $U_x$ and $U_y$ are continuous at $\{|y| = (p^*-1)|x|\}.$  This is easily verified since the partial derivatives are computed above (just switch the two pieces of each function).  Observe that for $x\neq 0$ and $y\neq 0, U_{x}<0$ and for $y\neq 0, U_{y}>0.$  Then the argument above showing $U(x_1, x_2) = U\left(x_3^{\frac 1p}, \sqrt{B_\tau^{\frac 2p}-\tau^2x_3^{\frac 2p}}\right)$ uniquely determines $B_\tau$ also holds for this range of $p-$values as well, except maybe at $x_1 = x_2 = 0.$  Suppose $U(0,0) = U\left(x_3^{\frac 1p}, \sqrt{B_\tau^{\frac 2p}-\tau^2x_3^{\frac 2p}}\right),$ then $B_\tau(x) = ((p^*-1)^2+\tau^2)^{\frac p2}x_3.$  So $B_\tau(x)$ is uniquely determined by the fixed $x-$value.  \qedhere
\end{proof}
\begin{cor}\label{Bellmancontinuity}
$B_\tau$ is continuous in $\Omega.$
\end{cor}

\begin{proof}
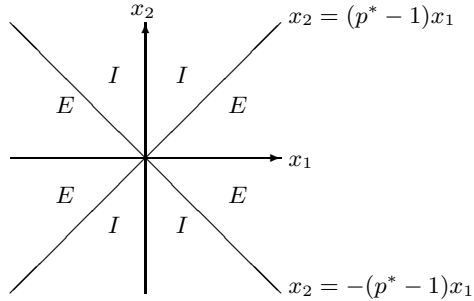
\begin{figure}[h]
%\begin{center}
\setlength{\unitlength}{0.02cm}
\begin{picture}(200,200)(0,0)
\thinlines
\put(90,10){\vector(0,1){180}} %x axis
\put(0,100){\vector(1,0){180}} %y axis
\put(90,100){\line(1,1){90}} %y = (p^*-1)x
\put(90,100){\line(-1,1){90}}
\put(90,100){\line(1,-1){90}} %y = -(p^*-1)x
\put(90,100){\line(-1,-1){90}}
%\put(20,100){\line(3,1){100}} %y2 = (p-2)/p y1
%\put(30,110){\line(0,-1){6.6}} %1st vertical line
%\put(40,120){\line(0,-1){13.3}} %2nd vertical line
%\put(50,130){\line(0,-1){20}}
%\put(60,140){\line(0,-1){26.7}}
%\put(70,150){\line(0,-1){33.3}}
%\put(80,160){\line(0,-1){40}}
%\put(90,170){\line(0,-1){46.7}}
%\put(100,180){\line(0,-1){53.3}}
%\put(110,190){\line(0,-1){60}}
\put(185,190){\footnotesize $x_2 = (p^*-1)x_1$}
\put(185,10){\footnotesize $x_2 = -(p^*-1)x_1$}
%\put(120,140){\footnotesize $y_2=\big(\frac{p-2}{p}\big)y_1$}
%\put(150,80){\footnotesize $$}
\put(80,195){\footnotesize $x_2$}
\put(185,95){\footnotesize $x_1$}

\put(145,70){\footnotesize $E$}
\put(145,130){\footnotesize $E$}
\put(30,70){\footnotesize $E$}
\put(30,130){\footnotesize $E$}

\put(110,150){\footnotesize $I$}
\put(65,150){\footnotesize $I$}
\put(110,50){\footnotesize $I$}
\put(65,50){\footnotesize $I$}

\end{picture}
\caption{Location of Implicit (I) and Explicit (E) part of $B_\tau$ for $2\leq p <\infty.$}
\label{1_2}
%\end{center}
\end{figure}

In this proof only we will revert back to the notation $U_{p, \tau},$ rather than $U,$ to make clear the distinction when $\tau = 0$ or $\tau \neq 0.$  We only consider $2<p<\infty$ as the dual range is handled identically.  By Proposition \ref{unifyingsolution}, we have that $B_\tau$ is the unique positive solution to \ref{unifyingrelation}.  Since this is true for all $\tau \in \R,$ then $B_0=\left(B_\tau^{\frac 2p}-\tau^2x_3^{\frac 2p}\right)^{\frac p2}$ on $|x_2| \geq (p^*-1)|x_1|,$ since $U_{p,\tau} =\left(1+ \frac{\tau^2}{(p^*-1)^2}\right)^{\frac{p-2}{2}}U_{p,0}.$  Equivalently, we have 
\be \label{Burkholderrelation}
B_\tau=\left(B_0^{\frac 2p}+\tau^2x_3^{\frac 2p}\right)^{\frac p2}.
\ee
Since $B_0$ was shown to be continuous in \cite{VaVo3} (pg. 26) then $B_\tau$ is also continuous on $|x_2| \geq (p^*-1)|x_1|,$ using the relation.  This takes care of the implicit part of $B_\tau.$  The explicit part of $B_\tau$ is clearly continuous on $|x_2| \leq (p^*-1)|x_1|.$ \qedhere
\end{proof}

\begin{lemma}\label{smoothnessonline}
Let $1<p<\infty$.  Then, $B_\tau\big|_L$ is $C^1-$smooth on $\Omega,$ where $L$ is any line in $\Omega.$
\end{lemma}

\begin{proof}
Since $B_\tau\big|_L$ is $C^2-$smooth on $\Omega_+,$ all that remains to be checked is the smoothness at the gluing and symmetry lines, i.e. at $\{x_1 = 0\}, \{x_2 = 0\}$ and $\{|x_2| = (p^*-1)|x_1|\}.$  Let $L = L(t), t \in \R,$ be any line in $\Omega$ passing through any of the planes in question, such that $L(0)$ is on the plane.  Now plug $L(t)$ into (\ref{unifyingrelation}) and differentiate with respect to $t.$  Let $t \rightarrow 0^+$ and $t \rightarrow 0^-$ and equate the two relations.  This gives 
\beno
\frac{d}{dt}B_\tau(L(t))\big|_{t = 0^-}= \frac{d}{dt}B_\tau(L(t))\big|_{t = 0^+}.\qedhere
\eeno 
\end{proof}

\begin{prop}\label{candidateconcavity}
(Restrictive Concavity) Let $1<p<2$ and $|\tau| \leq \frac 12$ or $2 \leq p < \infty$ and $\tau \in \R.$  Suppose $x^{\pm} \in \Omega$ such that $x = \al^+ x^+ + \al^- x^-, \al^+ +\al^- =1.$  If $|x_1^+ -x_1^-| = |x_2^+ -x_2^-|$ then $B_\tau(x) \geq \al^+ B_\tau(x^+) + \al^- B_\tau(x^-).$
\end{prop}

\begin{proof}
Recall that Propositions \ref{solution1_2and2_2glue} and \ref{solutionsfrom2_2and3_2}, together with the symmetry property of $B_\tau,$ establish this result everywhere, except at $\{x_1 = 0\}, \{x_2 = 0\}$ and $\{|x_2| = (p^*-1)|x_1|\}.$  Let $f(t) = B_\tau\big|_{L(t)},$ where $L$ is any line in $\Omega,$ such that $L(0) \in\{x_1 = 0\}, \{x_2 = 0\}$ or $\{|x_2| = (p^*-1)|x_1|\}.$  Since $f''<0$ for $t<0$ and $t>0$ and $f$ is $C^1-$smooth (by Lemma \ref{smoothnessonline}), then $f$ is concave.\qedhere
\end{proof}

\begin{prop}\label{Bellmanupperbound}
Let $1<p<\infty.$  If a function $\widetilde{B}$ has restrictive concavity and $\widetilde{B}_\tau(x_1, x_2, |x_1|^p) \geq (\tau^2 x_1^2 +x_2^2)^{\frac p2},$ then $\widetilde{B}_\tau \geq \B_\tau.$  In particular, $B_\tau \geq \B_\tau.$
\end{prop}

\begin{proof}
This was proven in \cite{VaVo3} for $B_0$ (Lemma 2 on page 29).  The same proof will apply here to $B_\tau.$ \qedhere
\end{proof}

\begin{prop}\label{Bellmanlowerbound}
For $1<p<\infty, B_\tau \leq \B_\tau.$
\end{prop}

\begin{proof}
For $1<p<2$ there is a direct proof, which will be discussed first.  By (\ref{Burkholderrelation}) we know that $B_0=\left(B_\tau^{\frac 2p}-\tau^2x_3^{\frac 2p}\right)^{\frac p2}$ on $\{|x_2| \leq (p^*-1)|x_1|\}.$  Consider, $\widetilde{\B}_0=\left(\B_\tau^{\frac 2p}-\tau^2x_3^{\frac 2p}\right)^{\frac p2}.$  It suffices to show that $B_0 \leq \widetilde{\B}_0.$  But, $B_0 = \B_0$ (as Burkholder showed), so without the supremum's we can reduce to simply showing
\beno
\langle |g|^p \rangle_I^{\frac 2p} + \tau^2\langle |f|^p \rangle_I^{\frac 2p} \leq \langle (\tau^2|f|^2+|g|^2)^{\frac p2} \rangle_I.
\eeno
Apply Minkowski:  $\left\| \int_I{(A, C)}\right\|_{l^{\frac 2p}} \leq \int_I{\|(A, C)\|_{l^{\frac 2p}}}.$  Choosing $A = |g|^p$ and $C = |\tau f|^p$ proves the result.  So we have shown that $B_\tau \leq \B_\tau$ on $\{|x_2| \leq (p^*-1)|x_1|\}.$  

Now we would like to show that $B_\tau \leq \B_\tau$ on $\{|x_2| \geq (p^*-1)|x_1|\}.$  Let $H_1(x_1, x_2, x_3) = B_\tau(x_1, x_2, x_3) - B_\tau(0,0,1)x_3.$  Lemma \ref{leastbiconcavemajorantlemma}, in the next section, proves that $H_1(x_1, x_2, \cdot)$ is an increasing function starting at $H_1(x_1, x_2, |x_1|^p) = v_\tau(x_1, x_2)$ and increasing to $\widetilde{U}_{p,\tau}(x,y):= \sup_{t\geq|x|^p}\{B_\tau(x,y,t)-B_\tau(0,0,1)t\}.$  The same proof works for $H_2(x_1, x_2, x_3) = \B_\tau(x_1, x_2, x_3) - \B_\tau(0,0,1)x_3.$  So
\beno
H_2(x_1, x_2, x_3) \geq v_\tau(x_1, x_2) = B_\tau(x_1, x_2, x_3) - B_\tau(0, 0, 1)x_3.
\eeno
Since $B_\tau(0,0,1) \leq \B_\tau(0,0,1),$ then $B_\tau \leq \B_\tau$ on $\{|x_2| \geq (p^*-1)|x_1|\}.$

Now we consider $2<p<\infty.$  Let $\e>0$ be arbitrarily small and consider the following extremal functions\begin{displaymath}
   f(x) = \left\{
     \begin{array}{lr}
       -c & : 1<x<\e\\
       \gamma f\left(\frac{t-\e}{1-2\e}\right) & : \e < x < 1-\e\\
       c & : 1-\e <x<1,
     \end{array}
   \right.
\end{displaymath} 
\begin{displaymath}
   g(x) = \left\{
     \begin{array}{lr}
       d_- & : 1<x<\e\\
       \gamma g\left(\frac{t-\e}{1-2\e}\right) & : \e < x < 1-\e\\
       d_+ & : 1-\e <x<1,
     \end{array}
   \right.
\end{displaymath} 
where $c, d_{\pm}$ and $\gamma$ are defined so that $f$ and $g$ are a pair of test functions at $(0, x_2, x_3).$  We can use $f$ and $g$ to show, just as in \cite{VaVo3} (Lemma 3, pg. 30), that 
\be \label{BellUppBnd}
B_\tau(0,x_2, x_3) \leq \B_\tau(0,x_2, x_3).
\ee

Now we need to take care of the estimate when $x_1 \neq 0.$  Making a change of coordinates from $x$ to $y$ we only need to consider $y \in \Xi_+,$ by the symmetry property of the Bellman function and Bellman function candidate.  So far we have that $M_\tau(\yone, \yone, \ythr) \leq \M_\tau(\yone, \yone, \ythr)$ by (\ref{BellUppBnd}).  The Dirichlet boundary conditions give that $M(\yone, \ytwo, (\yone - \ytwo)^{p}) = \M(\yone, \ytwo, (\yone - \ytwo)^{p}).$  On any characteristic in $\{\frac{p-2}{p}\yone \leq \ytwo \leq \yone\},$ see ~Figure ~\ref{gluedchar1_2and2_2}, $M_\tau$ is linear (since it is the Monge--Amp\`ere solution) and $\M_\tau$ is concave (by Proposition \ref{concavity}).  Therefore, $M_\tau(\yone, \ytwo, \ythr) \leq \M_\tau(\yone, \ytwo, \ythr)$ on $\{\frac{p-2}{p}\yone \leq \ytwo \leq \yone\}.$  For the remaining part of $\Xi_+,$ we can use the same proof as for $1<p<2$ to get $M_\tau(\yone, \ytwo, \ythr) \leq \M_\tau(\yone, \ytwo, \ythr)$ on $\{-\yone \leq \ytwo \leq \frac{p-2}{p}\yone\}.$\qedhere
\end{proof}

Now that we have proven $B = \B,$ we will mention another surprising fact.

\begin{defi}
We define $\B^l = \B^l(x_1, x_2, x_3)$ as the least restrictively concave majorant of $(x_2^2+\tau^2x_1^2)^{\frac p2}$ in $\Omega.$
\end{defi}

\begin{prop}
For $1< p < 2$ and $\tau \leq \frac 12$ or $2 \leq p < \infty$ and $\tau \in \R$ we have $B = \B = \B^l.$
\end{prop}

This is proven in \cite{BJV}.

\section{\bf Proving the main result}\label{sectionmainresult}
Now that we have the Bellman function, the main result can be proven without too much difficulty.  But first, we will find another relationship between $U$ and $v.$  Quite surprisingly, $U$ is the least zigzag-biconcave majorant of $v.$

\begin{defi}
A function of $(x,y)$ that is biconcave in $(x+y, x-y)$ we call zigzag-biconcave.
\end{defi}

\begin{lemma} \label{leastbiconcavemajorantlemma}
Let $1<p<\infty$ and $\widetilde{U}(x,y)= \sup_{t\geq|x|^p}\{B_\tau(x,y,t)-B_\tau(0,0,1)t\}.$  Fix $(x,y).$  The function $H(x,y,t) = B_\tau(x,y,t)-B_\tau(0,0,1)t$ is increasing in $t$ from $H(x,y,|x|^p)=v(x,y) :=(\tau^2|x|^2+|y|^2)^{\frac p2}-((p^*-1)^2+\tau^2)^{\frac p2}|x|^p$ to $\widetilde{U}_{p,\tau}(x,y).$ 
\end{lemma}
\begin{proof}
Recall that $B_\tau$ is continuous in $\Omega$ and for $(x,y)$ fixed, $B_\tau(x,y,\cdot)$ is concave.  Then $H(x,y,\cdot)$ is also concave.  Since $\widetilde{U}_{p,\tau}(x,y)= \sup_{t\geq|x|^p}\{B_\tau(x,y,t)-B_\tau(0,0,1)t\},$ then it either increases to $\widetilde{U}(x,y),$ or there exists $t_0$ such that $H(x,y,t_0) = \widetilde{U}(x,y)$ and $H$ is decreasing for $t>t_0.$  If $H$ is decreasing for $t>t_0,$ then $H \longrightarrow -\infty$ as $t \longrightarrow \infty$ by concavity.  Then there exists $\e >0$ and $t' > t_0$ such that $H(x,y,t') < \e t'.$  So we have, $\limsup_{t \to \infty}{\frac{H(x,y,t)}{t}}<-\e.$  But,
\beno
\lim_{t \to \infty}\frac{H(x,y,t)}{t} = \lim_{t \to \infty}\left[B_\tau\left(\frac{x}{t^{\frac 1p}}, \frac{y}{t^{\frac 1p}}, 1\right) - B_\tau(0,0,1)\right] = 0, 
\eeno
by continuity of $B_\tau$ at $(0,0,1).$  This gives us a contradiction.  Therefore, $H(x,y,t)\geq -\e t,$ for all $t$ and all $\e >0,$ i.e. $H$ is non-negative concave function on $[|x|^p,\infty).$  So $H(x,y,\cdot)$ is increasing and $H(x,y,|x|^p) = v_{p,\tau}(x,y)$ by the Dirichlet boundary conditions of $B_\tau$ in Proposition \ref{Bellmanprop}.\qedhere
\end{proof}

\begin{prop}\label{Uleastzigzagmajorant}
For $1<p<2$ and $\tau| \leq \frac 12$ or $2 \leq p < \infty$ and $\tau \in \R, U_{p,\tau}(x,y) = \widetilde{U}_{p,\tau}(x,y).$ 
\end{prop}

\begin{proof}
Suppose $2\leq p<\infty$ and $|y| \geq (p-1)|x|.$  Then 
\beno
\widetilde{U}_0(x,y) &=& \lim_{t \to \infty}\left(B_0(x,y,t)-B_0(0,0,1)t\right)\\
&=& \lim_{t \to \infty}\frac{B_0\left(\frac{x}{t^{\frac 1p}},\frac{y}{t^{\frac 1p}},1\right)-B_0(0,0,1)}{1/t}\\
&=& \frac{d}{du}B_0(u^{\frac 1p}x, u^{\frac 1p}y, 1)\bigg|_{u=0}.
\eeno
Now we repeat the same steps and obtain
\beno
\widetilde{U}_\tau(x,y) &=& \lim_{t \to \infty}\left(B_\tau(x,y,t)-B_\tau(0,0,1)t\right)\\
&=& \frac{d}{du}\left[\left(B_0^{\frac 2p}(u^{\frac 1p}x, u^{\frac 1p}y, 1)+\tau^2\right)^{\frac p2}\right]\bigg|_{u=0}\\
&=& \left[\!\!\left(\!\!B_0^{\frac 2p}(u^{\frac 1p}x, u^{\frac 1p}y, 1)+\tau^2\!\!\right)^{\!\!\!\frac{p-2}{2}}\!\!\!\!\!\!B_0^{\frac{2-p}{p}}(u^{\frac 1p}x, u^{\frac 1p}y, 1)\frac{d}{du}B_0(u^{\frac 1p}x, u^{\frac 1p}y, 1)\right]\!\!\bigg|_{u=0}\\
&=& \left(1+\frac{\tau^2}{(p-1)^2}\right)^{\frac{p-2}{2}}\widetilde{U}_0(x,y)\\
&=& \left(1+\frac{\tau^2}{(p-1)^2}\right)^{\frac{p-2}{2}}U_0(x,y),
\eeno
where the last equality is by \cite{Bu3}.  Therefore, $\widetilde{U}_\tau(x,y) = U_\tau(x,y).$

Now suppose $|y| \leq (p-1)|x|.$  Looking at the explicit form of $B_\tau$ in the region, note that $B_\tau(x,y,\cdot)$ is linear.  So $\widetilde{U}_\tau(x,y) = \sup_{t \geq |x|^p}\{B_\tau(x,y,t)-B_\tau(0,0,1)t\} = \sup_{t \geq |x|^p}\{B_\tau(x,y,0)\} = v_\tau(x,y) = U_\tau(x,y).$

We can apply the same proof to show that $\widetilde{U}_\tau(x,y) = U_\tau(x,y)$ for $1<p<2.$ \qedhere
\end{proof}

\begin{prop}
 $U$ is the least zigzag-biconcave majorant of $v$.
\end{prop}

Refer to \cite{BJV} for the proof.

We now have enough machinery to easily prove the main result, in terms of the Haar expansion of a $\R-$valued $L^p$ function.  
\begin{theorem}\label{realHaarmainresult}
Let $1<p<2, |\tau|\leq \frac 12$ or $2\leq p<\infty, \tau\in\R.$  Let $f, g:[0,1] \rightarrow \R.$  If $|\langle g \rangle_{[0,1]}| \leq (p^*-1)|\langle f \rangle_{[0,1]}|$ and $|(f, h_J)| = |(g, h_J)|$ for all $J \in \D,$then $\langle (\tau^2 |f|^2+|g|^2)^{\frac p2}\rangle_{[0,1]} \leq ((p^*-1)^2+\tau^2)^{\frac p2}\langle |f|^p \rangle_{[0,1]},$ where $((p^*-1)^2+\tau^2)$ is the sharp constant and $p^*-1 = \max\left\{p-1, \frac{1}{p-1}\right\}.$
\end{theorem}

\begin{proof}
Suppose that $2 \leq p<\infty$ and $\tau \in \R.$  The proof relies on the fact that the $B = \B$ (Propositions \ref{Bellmanupperbound} and \ref{Bellmanlowerbound}) and $\,U(x,y) = \sup_{t \geq |x|^p}\{B(x,y,t)-B(0,0,1)t\}$ (Proposition \ref{Uleastzigzagmajorant}).

Since $|y| \leq (p^*-1)|x|$ on $\Omega,$  then 
\beno
U(x,y) = v(x,y) = (|y|^2+\tau^2|x|^2)^{\frac p2}-((p^*-1)^2+\tau^2)^{\frac p2}|x|^p \leq 0.
\eeno
Then, 
\beno
\sup_{\substack{t>|x|^p\\ |y| \leq (p^*-1)|x|}}\{B(x,y,t)-B(0,0,1)t\} \leq 0.
\eeno
But, $U(0,0)=0,$ therefore 
\be \label{mainresulthaar}
\sup_{\substack{t>|x|^p\\ |y| \leq (p^*-1)|x|}}\frac{B(x,y,t)}{t} = B(0,0,1) = ((p^*-1)^2 + \tau^2)^{\frac p2}.
\ee
Observing the relationship $B = \B,$ gives the desired result.

For $1<p<2, |\tau|\leq \frac 12$ and $|y| \leq (p^*-1)|x|,$ 
\beno
U(x,y) = p\left(1-\frac{1}{p^*}\right)\left(1+\frac{\tau^2}{(p^*-1)^2}\right)^{\frac{p-2}{2}}(|x|+|y|)^{p-1}[|y|-(p^*-1)|x|] \leq 0,
\eeno
so we have (\ref{mainresulthaar}) by the same reasoning as for $2 \leq p<\infty$. \qedhere
\end{proof}

\begin{remark}
Note that Minkowski's inequality together with Burkholder's original result gives the same upper estimate for $2 \leq p < \infty.$  
\end{remark}

Indeed, if $f \in L^p[0,1]$ and $g$ is the corresponding martingale transform then Minkowski's inequality gives,
\beno
\|g^2 + \tau^2f^2\|_{L^{\frac p2}}^{\frac p2} & \leq & (\|g^2\|_{L^{\frac p2}}+ \|\tau^2 f^2\|_{L^{\frac p2}})^{\frac p2}  =  (\|g\|_{L^{p}}^2+ \|\tau f\|_{L^{p}}^2)^{\frac p2} \\
&\leq & \|f\|_{L^p}^p((p^*-1)^2+\tau^2)^{\frac p2}.
\eeno

This is very surprising in the sense that the ``trivial" constant $((p^*-1)^2+\tau^2)^{\frac p2}$ is actually the sharp constant.

Now we will prove the main result for Hilbert-valued martingales.  The same ideas can be used to extend the previous result to Hilbert-valued $L^p-$functions as well.  Let $\mathbb{H}$ be a separable Hilbert space with $\|\cdot\|_{\mathbb{H}}$ as the induced norm.

\begin{theorem}
Let $1<p<\infty, (W, \mathcal{F}, \mathbb{P})$ be a probability space and $\{f_k\}_{k \in \Z^+}, \{g_k\}_{k\in \Z^+}: W \rightarrow \mathbb{H}$ be two $\mathbb{H}-$valued martingales with the same filtration $\{\mathcal{F}_{k}\}_{k \in \Z^+}.$  Denote $d_k = f_k - f_{k-1}, d_0 = f_0, e_k = g_k - g_{k-1}, e_0 = g_0$ as the associated martingale differences.  If $\|e_k(\omega)\|_{\mathbb{H}} \leq \|d_k(\omega)\|_{\mathbb{H}},$ for all $\omega \in W$ and all $k \geq 0$ and $\tau \in [-\frac 12, \frac 12]$ then 
\beno
\left\|\left(\sum_{k=0}^n{e_k}, \tau \sum_{k=0}^n{d_k} \right)\right\|_{L^p(W, \mathbb{H}^2)} \leq ((p^*-1)^2 + \tau^2)^{\frac p2}\left\|\sum_{k=0}^n{d_k}\right\|_{L^p(W, \mathbb{H})}, 
\eeno
where 
$((p^*-1)^2 + \tau^2)^{\frac p2}$ is the best possible constant and $p^*-1 = \max\{p-1, \frac 1{p-1}\}.$  For $2 \leq p<\infty,$ the result is also true, with the best possible constant, if $\tau\in\R.$ 
\end{theorem}
In the theorem, ``best possible" constant means that if $C_{p,\tau} < ((p^*-1)^2+\tau^2)^{\frac 12},$ then for some probability space $(W, \mathcal{G}, P)$ and a filtration $\mathcal{F},$ there exists $\mathbb{H}-$valued martingales $\{f\}_k$ and $\{g\}_k,$ such that \beno
\left\|\left(g_k, \tau f_k \right)\right\|_{L^p([0,1], \mathbb{H}^2)} > C_{p,\tau}\left\|f_k\right\|_{L^p([0,1], \mathbb{H})}. 
\eeno

\begin{proof}
We will prove the result for $2 \leq p<\infty,$ since the result for $1<p<2$ is similar.  Replace $|\cdot|$ with $\|\cdot\|_{\mathbb{H}},$ in $U_{p, \tau}.$  Let $f_n = \sum_{k=0}^n{d_k}$ and $g_n = \sum_{k=0}^n{e_k}.$  Recall that $U := U_{p,\tau}$ is the least zigzag-biconcave majorant of $v := v_{p, \tau}.$  As in \cite{Bu4} (pages 77-79), 
\be \label{hilbertization}
U_{p, \tau}(x+h, y+k) \leq U_{p,\tau}(x,y)+ \Re(\partial_xU_{p,\tau}, h)+ \Re(\partial_yU_{p,\tau}, k),
\ee
for all $x, y, h, k \in \mathbb{H},$ such that $|k| \leq |h|$ and $\|x+ht\|_{\mathbb{H}}\|x+kt\|_{\mathbb{H}}>0.$  The result in (\ref{hilbertization}) follows from the zigzag-biconcavity and implies that $\mathbb{E}[U(f_k, g_k)]$ is a supermartingale.  Lemma \ref{leastbiconcavemajorantlemma} gives that $v(f_n, g_n) \leq U(f_n, g_n).$  Therefore, 
\beno
\mathbb{E}[v(f_n, g_n)] \leq \mathbb{E}[U(f_n, g_n)] \leq \mathbb{E}[U(f_{n-1}, g_{n-1})] \leq \cdots \leq  \mathbb{E}[U(d_0, e_0)].
\eeno
But, $\mathbb{E}[U(d_0, e_0)] \leq 0$ in both pieces of $U_\tau$ since $2-p^* \leq 0$ and $\|e_0\|_{\mathbb{H}} \leq \|d_0\|_{\mathbb{H}}.$  Thus, $\mathbb{E}[v_\tau(f_n, g_n)] \leq 0.$  The constant, in the estimate, is best possible, since it was attained in Theorem \ref{realHaarmainresult}. \qedhere
\end{proof}

\begin{remark}
For $1<p<2$ and $|\tau| > \frac 12,$ the ``trivial" constant $((p^*-1)^+\tau^2)^{\frac p2}$ in the main result is no longer sharp because of a ``phase transition".  In \cite{BJV} there is an $L^p-$function, $f,$ constructed so that together with it's martingale transform, $g,$ we have $\langle (\tau^2 |f|^2+|g|^2)^{\frac p2}\rangle_{[0,1]} > ((p^*-1)^2+\tau^2)^{\frac p2}\langle |f|^p \rangle_{[0,1]}$ for large $\tau.$
\end{remark}

\section{Addendum 1} \label{addendum1}

Throughout this Section the arguments may seem brief in comparison to Section \ref{solutionpgreaterthan2}.  The reason for this is because we cover the exact same argument as in Section \ref{solutionpgreaterthan2}, only with slightly different cases.  So if any arguments are unclear, then returning to Section \ref{solutionpgreaterthan2} should help to clear up any difficulties.  We will first consider Case $(3_2)$ to get a partial Bellman function candidate.
\subsection{\bf Considering Case $(3_2)$}\label{partial3_2subsec}

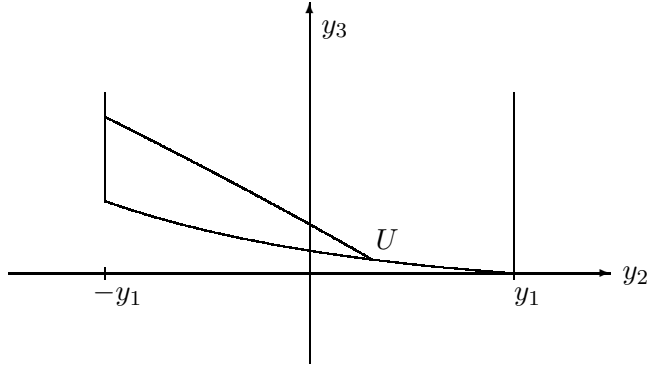
\begin{figure}[h]

\setlength{\unitlength}{0.8cm}

\begin{picture}(10,6)(-5,-1)
\put(-5,0){\vector(1,0){10}} %axis
\put(0,-1.5){\vector(0,1){6}} %axis
\put(5.2,-0.1){$y_2$} %label

\put(-3.6,-.45){$-y_1$} %label
\put(-3.4,-.1){\line(0,1){.2}} %tick

\put(1.1,.4){$U$} %label

\put(3.4, -.45){$y_1$} %label
\put(0.2,4){$y_3$} %label

\qbezier(3.4,0)(-0.8853,0.3)(-3.4,1.2) %parabola

\put(-3.4,1.2){\line(0,1){1.8}} %first vertical line
\qbezier(1,.25)(-1,1.4)(-3.4,2.6)
\put(3.4,-.1){\line(0,1){3.1}} %last vertical line

\end{picture}

\caption{Sample characteristic of Monge--Amp\`ere solution in Case $(3_2)$} 
\label{characteritic3_2}

\end{figure}

\begin{prop}\label{solution3_2}
For $1<p<\infty$ and $-\yone < \ytwo < \frac{2-p}{p}\yone$, $M$ is given implicitly by the relation $G(\yone -\ytwo, \yone +\ytwo) = \ythr G(1, \sqrt{\om^2 -\tau^2}).$
\end{prop}

This is proven through a series of Lemmas.

\begin{lemma}\label{betaminumtausquarelwrbnd}
$M(y) = t_2 y_2 + t_3 y_3 + t_0$ on the characteristic $y_2 dt_2 + y_3 dt_3 + dt_0 = 0$ can be simplified to $M(y) = \left(\frac{\sqrt{(\yone + u)^2 + \tau^2 (\yone -u)^2}}{\yone -u}\right)^p\ythr,$
where $u$ is the unique solution to the equation $\frac{\ytwo + (1-\frac{2}{p})\yone}{\ythr} = \frac{u + (1-\frac{2}{p})\yone}{(\yone -u)^p}$ and $-\yone < \ytwo < \frac{2-p}{p}\yone.$
\end{lemma}

\begin{proof}
Any characteristic, in Case$(3_2),$ goes from $U = (\yone, u, (\yone - u)^p)$ to $W = (\yone, -\yone, w).$  Recall the properties of the Bellman function we derived in Proposition \ref{Bellmanprop}, as we will be using them throughout the proof.  Using the Neumann property and the property from Proposition \ref{Pogorelov}, we get $M_{y_1} = -M_{y_2} = -t_2$ at $W.$  By homogeneity at $W$ we get 
\beno
-py_1 t_2 + pwt_3 + pt_0 = pM(W) = y_1 M_{y_1} + \ytwo M_{y_2} + p\ythr M_{y_3} = -2\yone t_2 + pwt_3.
\eeno

Now we follow the same idea as in Lemma \ref{reductionsolution1_2}, to get
$M(y) = \left(\frac{\sqrt{(\yone + u)^2 + \tau^2 (\yone -u)^2}}{\yone -u}\right)^p\ythr,$
where $u = u(\yone, \ytwo, \ythr)$ is the solution to the equation 

\be
\label{chreq3_2}
\frac{\ytwo + (1-\frac{2}{p})\yone}{\ythr} = \frac{u + (1-\frac{2}{p})\yone}{(\yone -u)^p}.
\ee

Fix $u = -(1-\twrp)\yone,$ then we see that $\ytwo = -(\twrp -1)\yone = u$ is also fixed by (\ref{chreq3_2}).  This means that the characteristics must lie in the sector shown in ~Figure ~\ref{3_2}, since they go from $U$ to $W \in \{y_2 = -y_1\}$.
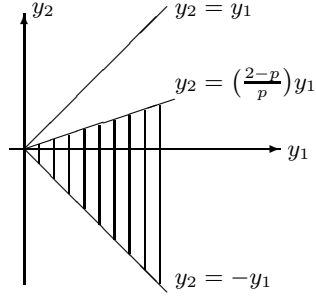
\begin{figure}[h]
%\begin{center}
\setlength{\unitlength}{0.02cm}
\begin{picture}(200,200)(0,0)
\thinlines
\put(20,10){\vector(0,1){180}} %x axis
\put(10,100){\vector(1,0){180}} %y axis
\put(20,100){\line(1,1){95}} %y2 = y1
\put(20,100){\line(1,-1){95}} %y2 = -y1
\put(20,100){\line(3,1){100}} %y2 = (p-2)/p y1
\put(30,103){\line(0,-1){12.4}} %1st vertical line
\put(40,106.1){\line(0,-1){25.4}} %2nd vertical line
\put(50,109.1){\line(0,-1){39}}
\put(60,112.6){\line(0,-1){52.1}}
\put(70,115.85){\line(0,-1){66}}
\put(80,119.1){\line(0,-1){79.1}}
\put(90,122.35){\line(0,-1){92}}
\put(100,125.6){\line(0,-1){105.6}}
\put(110,128.85){\line(0,-1){118.5}}
\put(120,190){\footnotesize $y_2=y_1$}
\put(120,10){\footnotesize $y_2=-y_1$}
\put(120,140){\footnotesize $y_2=\big(\frac{2-p}{p}\big)y_1$}
\put(150,80){\footnotesize $$}
\put(25,190){\footnotesize $y_2$}
\put(195,95){\footnotesize $y_1$}
\end{picture}
\caption{Range of characteristics in Case $(3_2)$ for $1<p<2.$}
\label{3_2}
%\end{center}
\end{figure}
The same argument as in Lemma \ref{reductionsolution1_2} can be used to verify that equation (\ref{chreq3_2}) has a unique solution in the sector $-\yone < \ytwo < \frac{2-p}{p}\yone.$
\end{proof}

\begin{lemma}
$M(y) = \left(\frac{\sqrt{(\yone + u)^2 + \tau^2 (\yone -u)^2}}{\yone -u}\right)^p\ythr$  can be rewritten as $G(\yone -\ytwo, \yone +\ytwo) = \ythr G(1, \sqrt{\om^2 -\tau^2})$ for $-\yone < \ytwo < \frac{2-p}{p} \yone.$ 
\end{lemma}
\begin{proof}
$\om = \left(\frac{M(y)}{\ythr}\right)^{\frac1p} = \frac{\sqrt{(\yone + u)^2 + \tau^2 (\yone -u)^2}}{\yone -u} \geq |\tau|$

Since $ \yone \pm u \geq 0$ and $\om^2 - \tau^2 \geq 0$, then $u = \frac{\sqrt{\om^2 - \tau^2}-1}{\sqrt{\om^2 - \tau^2}+1}\yone$ by inversion.  Substituting $u$ into $\frac{\ytwo + (1-\frac{2}{p})\yone}{\ythr} = \frac{u + (1-\frac{2}{p})\yone}{(\yone -u)^p}$ gives
\beno
2^{p-1}\yone^{p-1}[p\ytwo+(p-2)\yone] = \ythr(\sqrt{\om^2-\tau^2}+1)^{p-1}[\sqrt{\om^2-\tau^2}-(p-1)]
\eeno
or $(x_1 + x_2)^{p-1}[(p-1)x_2-x_1] = \left[\sqrt{B^{\frac2p} - \left(\tau x_3^{\frac1p}\right)^2}+x_3^\rp\right]^{p-1}\left[(p-1)\sqrt{B^{\frac2p} - \left(\tau x_3^{\frac1p}\right)^2}-x_3^\rp\right].$  Thus, $G(x_1, x_2) = G\left(x_3^\rp, \sqrt{B^{\frac2p} - \left(\tau x_3^{\frac1p}\right)^2}\right)$ or $G(\yone - \ytwo, \yone + \ytwo) = \ythr G(1, \sqrt{\om^2 - \tau^2}).\qedhere$

\end{proof}

As before, we must verify that this partial Bellman function candidate has the restrictive concavity property, so $\yone$ is no longer fixed.  To check restrictive concavity, we must show that $M_{\yone \yone} \leq 0, M_{\ytwo \ytwo} \leq 0, M_{\ythr \ythr} \leq 0$ and $D_1 \geq 0$ (note that $D_2 = 0$ by assumption).  These estimates are verified in the following series of lemmas.   

\begin{lemma}\label{H''3_2}
In Case $(3_2)$ we choose $H(\yone, \ytwo) = G(y_1-y_2,y_1+y_2)$ because of how the implicit solution is defined and obtain $\sign H'' = -\sign (p-2).$
\end{lemma} 

\begin{proof}

We already computed 
   \begin{displaymath}
   H'' = \left\{
     \begin{array}{lr}
       4G_{z_1 z_2}, \al_j = \beta_j\\
       0, \al_j = -\beta_j
     \end{array}
   \right.
\end{displaymath} 

in Lemma \ref{H''1_2}.  Since, $\alpha_1 = 1, \alpha_2 = -1, \beta_1 = 1$ and $\beta_2 = 1$ then $G_{\!z_{_1}\!z_{_2}}=-p(p-1)(p-2)(y_1+y_2)(2y_1)^{p-3}.$  \qedhere
\end{proof}

\begin{remark} \label{lowerbndbeta3_2}
In Case $(3_2), \beta > \frac{1}{p-1}$ in the sector $-\yone < \ytwo < \frac{2-p}{p}\yone,$ where $\beta := \sbeta.$  Equivalently, $B(x_1,x_2,x_3) \geq ((p^*-1)^2+\tau^2)^{\frac p2}x_3$ in $-\yone < \ytwo < \frac{2-p}{p}\yone.$
\end{remark}
This is trivial since
\beno
(\beta + 1)^{p-1}[1-(p-1)\beta] = G(1,\beta) & = & \frac{1}{\ythr}G(\yone-\ytwo, \yone+\ytwo)\\ 
& = & (2y_1)^{p-1}[(p-2)y_1 + py_2]<0.\\
\eeno

Now we have enough information to check the sign of $D_1.$  We will start limiting the values of $\tau,$ since it will be essential for having the restrictive concavity of the parital Bellman candidate from Case $(2_2)$ (see Remark \ref{limittauremark}).

\begin{lemma}\label{D13_2}
$D_1 > 0$ in Case $(3_2)$ for all $|\tau| \leq 1.$
\end{lemma}

\begin{proof}
We use the partial derivatives of $G$ computed in the proof of Lemma \ref{H''1_2} to make the computations of $\Phi'$ and $\Phi''$ easier.
\be 
\Phi(\om) & = & G(1, \beta)\notag \\ \label{Phiofomega3_2}
\Phi'(\om) & = & -p(p-1) \om [\beta + 1]^{p-2}\\ \notag 
\Phi''(\om) & = & -\frac{p(p-1)(1+\beta)^{p-3}}{\beta} \left[\beta(1+\beta)+(p-2)\om^2\right] \\ \notag
\Lambda & = & (p-1)\Phi' - \om \Phi''\\\notag
& = & -p (p-1)^2 \om(\beta + 1)^{p-2}+ \frac{p(p-1)\om (1+\beta)^{p-3}}{\beta}\left[\beta(1+\beta)+(p-2)\om^2\right]\\ \notag
& = & \frac{p(p-1)\om(1+\beta)^{p-3}}{\beta}\left[-(p-1)(1+\beta)\beta+\beta(1+\beta)+(p-2)\om^2\right]\\
& = & -\frac{p(p-1)(p-2)\om(1+\beta)^{p-3}[\beta -\tau^2]}{\beta}\label{Biglambda3_2}
\ee

Now we need to determine the sign of $\beta - \tau^2$ is for $1<p<\infty.$ By Remark \ref{lowerbndbeta3_2}, $\beta > \frac{1}{p-1} \geq \tau^2$ for $|\tau| \leq 1$ and $1<p<2.$ But, what about $p>2$? Using the form of the solution, $M,$ in Lemma \ref{betaminumtausquarelwrbnd} we obtain $(\yone + u)^2 + \tau^4(1-\tau^2)(\yone-u)^2>0$
\beno
& \Longleftrightarrow \frac{(\yone +u)^2+\tau^2(\yone-u)^2}{(\yone-u)^2}\geq\tau^2(1+\tau^2)\\
& \Longleftrightarrow \omega = \left(\frac{M(y)}{\ythr}\right)^{\frac 2p}>\tau^2(1+\tau^2)\\
& \Longleftrightarrow\beta -\tau^2>0,
\eeno
where $u$ is the unique solution to $\frac{\ytwo+(1-\frac 2p)\yone}{\ythr}=\frac{u+(1-\frac 2p)\yone}{\ythr}$ and $|\tau| \leq 1.$   Thus, $\sign D_1 = \sign H'' \sign \Lambda =[-\sign(p-2)]^2$ by (\ref{Di}) and Lemma \ref{H''3_2}. \qedhere
\end{proof}

The following lemma restricts the $p-$values for which our solution is a Bellman function candidate to $1<p<2.$
\begin{lemma}\label{checkconcavity3_2}
$\sign M_{\yone \yone} = \sign M_{\ytwo \ytwo} = \sign M_{\ythr \ythr} = \sign(p-2)$ in Case $(3_2)$ for all $|\tau| \leq 1.$  Consequently, $M$ is a Bellman function candidate for $1<p<2$ but not for $2<p<\infty,$ since it wouldn't satisfy the restrictive concavity needed.
\end{lemma}

\begin{proof}
By (\ref{M33}), (\ref{Phiofomega3_2}), (\ref{Biglambda3_2})), 
\beno
M_{\ythr \ythr} = \frac{p \om^{p-2}R_1^2H^2}{y_3^3}\left[\frac{\Lambda}{\Phi'}\right],\\
\eeno
giving $\sign M_{\ythr \ythr} = (-1)[-\sign(p-2)].$  By (\ref{Mii}), for $i = 1,2,$
\beno
M_{y_i y_i} &=& \frac{p \omega^{p-2}R_1}{\ythr}\left[(\omega R_2 + (p-1)R_1)(H')^2+\omega \ythr H''\right]\\
&=& \frac{p \omega^{p-2}}{\ythr (\Phi')^3}\left[\Lambda(H')^2+\omega \ythr H''(\Phi')^2\right],\\
\eeno
giving $\sign M_{y_i y_i} = (-1)[-\sign(p-2)],$ since $\Phi' < 0.$ \qedhere
\end{proof}
Now that we have a partial Bellman function candidate for $1<p<2,$ from Case $(3_2),$ satisfying all of the properties of the Bellman function, including restrictive concavity, we can turn our attention to Case $(2_2).$  From Case $(2_2)$ we will get a Bellman candidate on all $\Xi_+,$ or part of it, depending on the $\tau-$ and $p-$values.  The partial Bellman candidate, from Case $(2_2),$ turns out to be the missing half for Case $(3_2).$  We already have the solution for Case $(2)$ from Lemma \ref{case2solution}, but the value of the constant is needed before we can progress further.

\subsection{\bf Case (2) for $1<p<2$}\label{partial2_2subsec}

\begin{lemma}\label{constantfor2_2whenplessthan2}
If $1<p<2$ then in Case $(2_2),$ the value of the constant in Lemma \ref{case2solution} is $c = \left(\frac{1}{(p-1)^2} + \tau^2\right)^{\frac p2}.$
\end{lemma}

\begin{proof}
If $M(y) = (1+\tau^2)^{\frac p2}[y_1^2 + 2 \gamma y_1 y_2 + y_2^2]^{\frac p2} + c [y_3 - (y_1 - y_2)^p]$ (where $\gamma = \frac{1-\tau^2}{1+\tau^2}$) is to be a candidate or partial candidate, then it must agree, at $y_2 = \frac{2-p}{p}\yone,$ with the solution $M$ given implicitly by the relation $G(\yone -\ytwo, \yone +\ytwo) = \ythr G(1, \sqrt{\om^2 -\tau^2})$, from Proposition \ref{solution3_2}.  At $y_2 = \frac{2-p}{p}\yone,$ 
\beno
(\sbeta + 1)^{p-1}[1-(p-1)\sbeta] &=& G(1,\sbeta)\\ 
&=& \frac{1}{\ythr}(2y_1)^{p-1}[(2-p)y_1 + (p-2)y_1] = 0.\\  
\eeno
Since $\sbeta + 1 \neq 0$ then $\sbeta = \frac{1}{p-1},$ which implies $\om = \left(\frac{1}{(p-1)^2} + \tau^2\right)^{\frac 12}.$  So,  
\beno
\left(\frac{1}{(p-1)^2} + \tau^2\right)^{\frac p2} \ythr &=& \om^p \ythr\\ 
&=& M(\yone, \frac{2-p}{p}\yone, \ythr)\\
&=& \left[\left(\frac 2p y_1\right)^2 + \tau^2\left(\frac{2(p-1)}{p} \yone\right)^2\right]^{\frac 2p} + c\left[\ythr - \left(\frac{2(p-1)}{p} \yone\right)^p\right].\\
\eeno
Now just solve for $c.$ \qedhere
\end{proof}

In the following Lemma the value of $\tau$ has to be restricted to $|\tau| \leq 1,$ so that restrictive concavity is satisfied for our Bellman candidate.  Actually, the $\tau-$values play an even bigger role.  Depending on the value of $(\tau, p) \in [-1,1]\times (1,2),$ there is either one or two Bellman function candidates.  For $(\tau, p) \in B,$ from Figure \ref{signMy2y2pless2}, there is a partial Bellman candidate arising from Case $(2_2)$.  So we can glue this together with the other partial candidate obtained in Case $(3_2).$  This gives a Bellman candidate, as before, having characteristics as in ~Figure ~\ref{gluedchar2_2and3_2}.   For $(\tau, p) \in A \cup C$ the candidate obtained from Case $(2_2)$ maintains restrictive concavity throughout $\Xi_+$ and is therefore requires no gluing.  To avoid the difficulty of determining which candidate to choose and how to determine the optimal constant from Case $(2)$, we restrict $(\tau, p)$ to region $B$, or require that $|\tau| \leq \frac 12.$  
\begin{figure}[h]

\setlength{\unitlength}{0.8cm}

\begin{picture}(9,6)(-5,-1)
\put(-4,0){\vector(1,0){9}} %axis
\put(5.2,-0.1){$\tau$} %label
\put(0,-1){\vector(0,1){6}} %axis
\put(0.2,5.1){$p$} %label

\put(0.2,4.1){$2$} %label
\put(0.2,2.1){$1.5$} %label
\put(4.1,-0.5){$1$} %label
\put(2.1,-0.5){$0.5$} %label
\put(-4.1,-0.5){$-1$} %label
\put(-2.1,-0.5){$-0.5$} %label
\put(-3.4,2.1){$A$}%label
\put(3.2,2.1){$C$}%label
\put(-1.5,1.3){$B$} %label
\put(0.85,1.3){$B$} %label

\put(2,-.1){\line(0,1){.2}}
\put(-2,-.1){\line(0,1){.2}}

\put(-0.1,2){\line(1,0){.2}}

\qbezier(4,0)(2,0)(2,4) %parabola
\qbezier(-4,0)(-2,0)(-2,4) %parabola

\put(4,-.1){\line(0,1){4.1}} % vertical lines
\put(-4,-.1){\line(0,1){4.1}}
\put(-4,4){\line(1,0){8}}%horizontal line

\end{picture}

\caption{Splitting $[-1, 1] \times (1,2)$ in the $(\tau \times p)$--plane into three regions $A, B$ and $C.$  The curves separating regions $A, B$ and $C$ are where $M_{\ytwo \ytwo} =0$ in Case $(2_2).$} 
\label{signMy2y2pless2}

\end{figure}
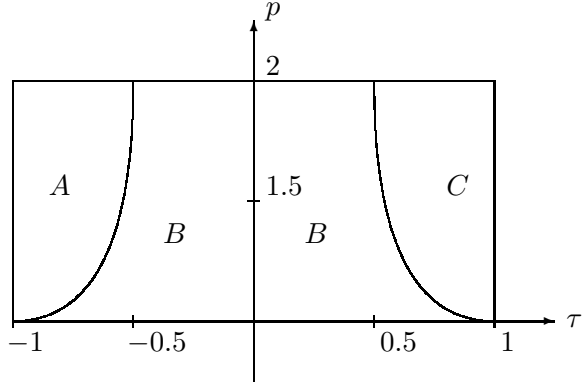
Recall that the partial Bellman candidate, $M,$ obtained from Case $(2_2),$ for $1<p<2,$ satisfies $M_{y_i \ythr} = M_{\ythr \ythr} = 0$ and hence $D_i =0,$ for $i = 1,2.$  So all that still needs to be checked for restrictive concavity is the sign of $M_{\yone \yone}$ and $M_{\ytwo \ytwo}.$  Since $M_{\yone \yone} \leq M_{\ytwo \ytwo}$, then we just need to show that $M_{\ytwo \ytwo} \leq 0$ on $\frac{2-p}{p}\yone \leq \ytwo \leq \yone$ in $\Xi_+.$  This is considered in the following Lemmas. 

\begin{lemma}\label{case2_2partialswhenplessthan2}
In Case $(2_2), M_{\ytwo \ytwo}(\yone, \frac{2-p}{p}\yone, \ythr) \leq 0$ for $|\tau|\leq 1$ and $1<p<2.$
\end{lemma}

\begin{proof}
The solution $M$ that we get from $(2_2),$ when $1<p<2,$ is obtained from Lemmas \ref{case2solution} and \ref{constantfor2_2whenplessthan2}.  Let $\gamma = \frac{1-\tau^2}{1+\tau^2}, f_1(y) = y_1^2+y_2^2+2\gamma\yone \ytwo, f_2(y) = (p-2)(\ytwo + \gamma\yone)^2 + f_1(y)$ and $f_3(y) = y_1-y_2.$  Then 
\beno
M_{\ytwo \ytwo} = p(1+\tau^2)^{\frac p2} f_1^{\frac{p-4}2}f_2 - p(p-1)\left(\frac 1{(p-1)^2}+\tau^2\right)^{\frac p2} f_3^{p-2}.
\eeno

By direct calculations one can verify, $M_{\ytwo \ytwo}(\yone, \frac{2-p}{p}\yone, \ythr) \leq 0$ when $|\tau|\leq 1.$ \qedhere
\end{proof}

\begin{remark}\label{limittauremark}
Note that Lemma \ref{case2_2partialswhenplessthan2} is false for $p$ near $2$ when $|\tau|$ larger than 1, so we cannot take a larger value and still maintain the restrictive concavity.  
\end{remark}

\begin{lemma}\label{case2_2glueingto3_2}
In Case $(2_2), M_{\ytwo \ytwo}(\yone, c\yone, \ythr) \leq 0,$ for all $c \in \left[\frac{2-p}{p}, 1\right].$  
\end{lemma}

\begin{proof}
Using $M_{\ytwo \ytwo}$ from Lemma \ref{case2_2partialswhenplessthan2} we see that $M_{\ytwo \ytwo} \leq 0$ is equivalent to 
\beno
(1+\tau^2)^{\frac p2}f_3^{2-p}\frac{f_2}{f_1^{\frac {4-p}2}}- p(p-1)\left(\frac 1{(p-1)^2}+\tau^2\right)^{\frac p2} \leq 0
\eeno
Observe that the function $f_2/(f_1^{\frac{4-p}{2}})$ is strictly positive, has a horizontal asymptote at the $y_2$-axis, increases on $(-\infty, -\gamma)$, and decreases on $(-\gamma, \infty).$  As $\ytwo$ increases from $\frac{2-p}p
$ to $1$, $f_3^{2-p}$ and $\frac{f_2}{f_1^{\frac 2{4-p}}}$ both decrease.  Since $M_{\ytwo \ytwo}(\yone, \frac{2-p}{p}\yone, \ythr) \leq 0,$ as shown in Lemma \ref{case2_2partialswhenplessthan2}, the result follows. \qedhere
\end{proof}

\begin{lemma}\label{firstsolutionplessthan2}
The Monge--Amp\`ere solution in Case $(2_2)$ yields the following results for $1<p<2.$  $M_{\ytwo \ytwo}(\yone, \yone, \ythr) < 0$ for $|\tau| \leq 1$ and $M_{\ytwo \ytwo}(\yone, -\yone, \ythr)>0$ for $|\tau| \leq \frac 12$
\end{lemma}

\begin{proof}
Let $f_1, f_2$ and $f_3$ be as in Lemma \ref{case2_2partialswhenplessthan2} and 
\beno
g = (1+\tau^2)^{\frac p2}f_3^{2-p}f_2 - (p-1)\left(\frac 1{(p-1)^2}+\tau^2\right)^{\frac p2}f_1^{\frac{4-p}{2}}.
\eeno
Note that $M_{\ytwo \ytwo}$ and $g$ have the same signs.  It is clear that $g(\yone, \yone, \ythr) < 0,$ proving the first inequality.  One can now verify that $g(\yone, -\yone, \ythr) > 0$ for $|\tau| \leq \frac 12$ which proves the second inequality.  
\end{proof}  

\begin{remark}One can see in the graph of $\frac{1}{\yone^{p-2}}g(\yone, \yone, \ythr)$ that $g(\yone, \yone, \ythr)<0,$ in regions $A$ and $C,$ (see Figure \ref{signMy2y2pless2}).  This tells us that the Bellman candidate from Case $(2_2)$ will maintain restrictive concavity throughout the domain in for $(\tau, p) \in A \cup C.$  Furthermore, there will be an improvement in the constant $((p^*-1)^2+\tau^2)^{\frac p2}$ that can still be used to still maintain restrictive concavity in $A \cup C.$  
\end{remark}

By Lemmas \ref{case2_2glueingto3_2} and \ref{firstsolutionplessthan2} we obtain a partial Bellman candidate from Case $(2_2),$ when $1<p<2$ and $|\tau| \leq \frac 12.$  As before, we will glue this partial candidate from Case $(2_2)$ to the partial candidate in Case $(3_2)$ to obtain the Bellman candidate for $1<p<2.$

\section{Addendum 2} \label{addendum}

Now that we have particular cases in which the Monge--Amp\`ere solution gives a Bellman function candidate, we would like to discuss the remaining cases.  It can be shown that all remaining cases do not yield a Bellman function candidate, except for Case $(4)$ which is still not determined.
\subsection{\bf Case $(1_2)$ for $1<p<2$ and Case $(3_2)$ for $2<p<\infty$ do not lead to a Bellman candidate}
It was shown in Lemmas \ref{My3y3pgreaterthan2}, \ref{checkconcavity3_2} that the Monge--Amp\`ere solution obtained in each case does not have the appropriate restrictive concavity property to be a Bellman function candidate.  We mention this here again simply for clarity.
\subsection{\bf Case $(1_1)$ does not give a Bellman candidate}\label{subsec1_1nosolution}
We can consider Cases $(1_1)$ and $(3_1)$ simultaneously, for part of the calculation, since the same argument will work in both cases.  In both cases, $y_2$ is fixed and the Monge--Amp\`ere solution is given by $M(y) = t_1 y_1 + t_3 y_3 + t_0$ on the characteristics $dt_1 y_1 + dt_3 y_3 + dt_0 = 0.$  As shown in Figure \ref{characteristic1_1and3_1}, $\ytwo \geq 0$ in case $(1_2)$ and $\ytwo \leq 0$ in Case $(3_2),$  since if not then the characteristics go outside of the domain $\Xi_+.$
\begin{figure}[h]
  \centering
 \setlength{\unitlength}{0.8cm}
\subfigure[Case $(1_1)$]{
\begin{picture}(6,6)(0,-1)
\put(-1,0){\vector(1,0){5}} %axis
\put(0,-1.5){\vector(0,1){6}} %axis
\put(4.2,-0.1){$y_1$} %label
\put(2,.52){$U$} %label
\put(1, -.45){$y_2$} %label
\put(0.2,4){$y_3$} %label

\qbezier(1,0)(2,.5)(3.4,3) %parabola

\put(2,.9){\line(-1,1){1}} %characteristic
\put(1,-.1){\line(0,1){3.1}} %last vertical line
\end{picture}}                
\subfigure[Case $(3_1)$]
{\setlength{\unitlength}{0.8cm}
\begin{picture}(6,6)(0,-1)
\put(-1,0){\vector(1,0){6}} %axis
\put(0,-1.5){\vector(0,1){6}} %axis
\put(5.2,-0.1){$y_1$} %label
\put(2,1.52){$U$} %label
\put(0.8, -.45){$-y_2$} %label
\put(1,-.1){\line(0,1){.2}} %tick
\put(0.2,4){$y_3$} %label

\qbezier(1,1)(2,1.5)(3.4,4) %parabola

\put(2,1.9){\line(-1,1){1}} %characteristic
\put(1,1){\line(0,1){3.1}} %last vertical line

\end{picture}}
\caption{Sample characteric for Monge--Amp\`ere solution in Cases $(1_1)$ and $(3_1)$}
\label{characteristic1_1and3_1}
\end{figure}
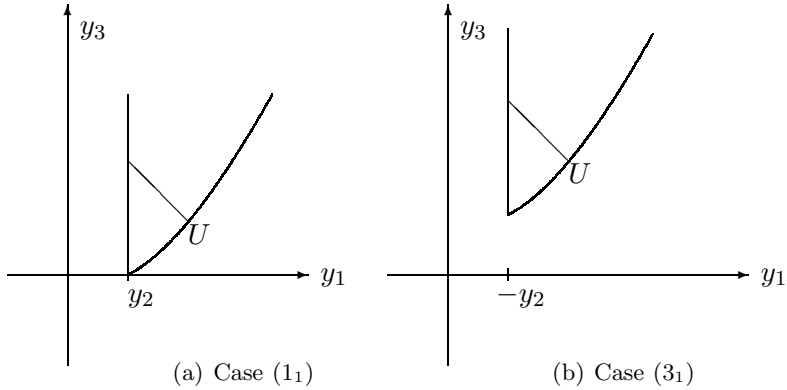

\begin{lemma}\label{solution1_1and3_1}
In Cases $(1_1)$ and $(3_1),$ the solution to the Monge--Amp\`ere can be written as, 
\beno
M(y) = \left(\frac{\sqrt{(u+\ytwo)^2+\tau^2(u-\ytwo)^2}}{u-\ytwo}\right)^p \ythr
\eeno
where $u = u(\yone, \ytwo, \ythr)$ is the solution to the equation $\frac{\yone + \left(\frac 2p -1\right)|\ytwo|}{\ythr} =  \frac{u + \left(\frac 2p -1\right)|\ytwo|}{(u-\ytwo)^p}.$
\end{lemma}

\begin{proof}
Any characteristic, in Cases $(1_1)$ and $(3_1)$, go from $U = (u, \ytwo, (u-\ytwo)^p)$ to $W = (|\ytwo|, \ytwo, w).$  Throughout the proof we will use the properties of the Bellman function derived in Proposition \ref{Bellmanprop}.  Using the Neumann property and the property from Proposition \ref{Pogorelov} we get $\ytwo M_{y_2} = \yone M_{y_1} = |\ytwo|t_1$ at $W.$  By homogeneity at $W$ we get 
\beno
p|y_2|t_1 + pwt_3 + pt_0 = pM(W) = y_1 M_{y_1} + \ytwo M_{y_2} + p\ythr M_{y_3} = 2t_1|y_2| + pwt_3
\eeno

Following the same argument as in Lemma \ref{reductionsolution1_2}, gives
$M(y) = \left(\frac{\sqrt{(u+\ytwo)^2+\tau^2(u-\ytwo)^2}}{u-\ytwo}\right)^p \ythr,$
where $u = u(\yone, \ytwo, \ythr)$ is the solution to the equation 

\be
\label{chreq1_1and3_1}
\frac{\yone + \left(\frac{2}{p}-1\right)|\ytwo|}{\ythr} = \frac{u + \left(\frac{2}{p}-1\right)|\ytwo|}{(u - \ytwo)^p}.
\ee
Since the solution, $M,$ does not satisfy the restrictive concavity property necessary to be the Bellman function (as we will soon show), we are not concerned about existence of the solution $u$ in equation (\ref{chreq1_1and3_1}). \qedhere
\end{proof}

\begin{lemma}\label{sgnytwo}
If $\om = \left(\frac{M(y)}{\ythr}\right)^{\frac 1p},$ then in Cases $(1_1)$ and $(3_1)$, the solution $u$ to equation (\ref{chreq1_1and3_1}) can be expressed as $u = \frac{\sbeta +1}{\sbeta -1}\ytwo$ and equation (\ref{chreq1_1and3_1}) can be rewritten as 
\be \label{halfimplicitsolution1_1and3_1}
2^p|\ytwo|^{p-1}[p\yone + (2-p)|\ytwo|] = \ythr |\beta -1|^{p-1}[p(\beta +1)+(2-p)|\beta-1|],
\ee
where $\beta = \sbeta.$  Furthermore, $\sign \ytwo = \sign (\beta -1).$
\end{lemma}

\begin{proof}
Let us show that $u = \frac{\sbeta +1}{\sbeta -1}\ytwo$ first.  This follows from inverting 
\beno
\om = \frac{\sqrt{(u+\ytwo)^2+\tau^2(u-\ytwo)^2}}{u-\ytwo},
\eeno
and using the properties $\om \geq |\tau|$ and $u \pm \ytwo \geq 0.$  Now that we have $u = \frac{\sbeta +1}{\sbeta -1}\ytwo, $ we can use it to get the next result.  Note that $u \geq 0$ and $\sbeta \geq 0,$ which implies that $\sign \ytwo = \sign (\sbeta -1).$  To get (\ref{halfimplicitsolution1_1and3_1}), simply plug $u = \frac{\sbeta +1}{\sbeta -1}\ytwo$ in equation (\ref{chreq1_1and3_1}). \qedhere
\end{proof}
We can no longer discuss Cases $(1_1)$ and $(3_1)$ together, so for the remainder of the Subsection the focus will be on Case $(1_1)$ only.

\begin{lemma} \label{partialsolutioncase1_1}
In Case $(1_1),$ the solution $M$ from Lemma \ref{solution1_1and3_1} can be rewritten in the  implicit form $G(\ytwo + \yone, \ytwo -\yone) = \ythr G(\sbeta, -1),$ where $G(z_1, z_2) = (z_1 + z_2)^{p-1}[z_1 - (p-1)z_2].$
\end{lemma}

\begin{proof}
Recall that for Case $(1_1)$ we have $\ytwo > 0.$ 
\beno
\ytwo = \frac 12 (x_2 -x_1)>0 \quad\implies\quad x_2 > x_1
\eeno
\beno
\sign (\sbeta -1) = \sign \ytwo > 0 \quad\implies\quad \sbeta > 1 \quad\implies\quad \om > \sqrt{\tau^2 + 1}
\eeno
So, $\B(x) = M(y) > \ythr (\tau^2 +1)^{\frac p2}.$  Now (\ref{halfimplicitsolution1_1and3_1}) can be rewritten as
\beno
(x_2 -x_1)^{p-1}[(p-1)x_1 + x_2]=\left[\sqrt{\B^{\frac 2p}-\tau^2 x_3^{\frac 2p}} - x_3^{\rp}\right]^{p-1}\left[\sqrt{\B^{\frac 2p}-\tau^2 x_3^{\frac 2p}} + (p-1)x_3^{\rp}\right].
\eeno
Therefore, 
\beno
G(x_2, -x_1) = G\left(\sqrt{\B^{\frac 2p}-\tau^2 x_3^{\frac 2p}}, -x_3^{\rp}\right)
\eeno
or by factoring out $x_3^{\rp}$ on the right side we get
\beno
G(\ytwo + \yone, \ytwo -\yone) = \ythr G(\sbeta, -1). \qedhere
\eeno
\end{proof}

Recall that the Monge--Amp\`ere solution must satisfy the restrictive concavity conditions in Proposition \ref{D_i} to be a Bellman function candidate.  We will show that the Monge--Amp\`ere solution obtained in Case $(1_1)$ has $D_1<0$ and therefore cannot be a Bellman candidate. 
 
\begin{lemma}\label{H''1_1}
In Case $(1_1)$ we choose $H(\yone, \ytwo) = G(y_1+y_2,-y_1+y_2)$ because of how the implicit solution is defined and obtain $\sign H'' = \sign (p-2)$
\end{lemma} 

\begin{proof}

We already computed 
   \begin{displaymath}
   H'' = \left\{
     \begin{array}{lr}
       4G_{z_1 z_2}, \al_j = \beta_j\\
       0, \al_j = -\beta_j
     \end{array}
   \right.
\end{displaymath} 

in Lemma \ref{H''1_2}.

Since, $\alpha_1 = 1, \alpha_2 = 1, \beta_1 = -1$ and $\beta_2 = 1$ then $G_{\!z_{_1}\!z_{_2}}=p(p-1)(p-2)(y_1-y_2)(2y_2)^{p-3}.$ \qedhere
\end{proof}

\begin{lemma}\label{D21_1}
If $p \neq 2$ then $D_2 < 0$ in Case $(1_1)$ for all $\tau.$
\end{lemma}

\begin{proof}
We use the partial derivatives of $G$ from the proof of Lemma \ref{H''1_2} to make the computations of $\Phi'$ and $\Phi''$ easier.  Let $\al_p = \frac{p (p-1) \om (\beta -1)^{p-3}}{\beta^3}$ and $\beta = \sqrt{\om^2 - \tau^2}.$
\beno
\Phi(\om) & = & G(\beta,1)\notag \\ 
\Phi'(\om) & = & p[\beta - 1]^{p-2}[\om + (p-2)\om \beta^{-1}]\\ 
\Phi''(\om) & = & G_{z_1 z_2}(\beta, -1) \beta^{-2}\om^2 + G_{z_1}(\beta, -1)[-\om \beta^{-3} + \beta^{-1}]  \\ 
& = & p (p-1)[\beta -1]^{p-3}[\beta + p-3]\frac{\om^2}{\beta^2} - p \frac{\tau^2}{\beta^3}[\beta -1]^{p-2}[\beta + p-2]\\ 
\Lambda & = & (p-1)\Phi' - \om \Phi''\\
& = & \al_p [(\beta -1)\beta^3(\beta^{-1}(p-2) +1)-\om^2\beta(\beta+p-3) + \tau^2(\beta-1)(\beta+p-2)]\\ 
& = & \al_p[(\beta^2 + \tau^2)(\beta -1)(\beta + p-2) - \om^2 \beta (\beta + p-3)] \\ 
& = & \al_p \om^2[\beta^2 + \beta(p-2) - \beta -(p-2) - \beta^2 - \beta (p-3)]\\
& = & -\frac{p(p-1)(p-2) \om^3 (\sbeta -1)^{p-3}}{(\sbeta)^3}
\eeno
From Lemma \ref{sgnytwo}, $\sign (\beta-1) = \sign \ytwo >0$ and $\om^2 > \tau^2 > 0.$  Therefore, by Lemma \ref{H''1_1} and (\ref{Di}) $\sign D_2 = \sign H'' \sign \Lambda = - (\sign(p-2))^2 < 0.$ \qedhere
\end{proof}

Since $D_2<0$ in Case $(1_1)$ then we get the following result.

\begin{prop}
Case $(1_1)$ does not give a Bellman function candidate.
\end{prop}

\subsection{\bf Case $(3_1)$ does not provide a Bellman function candidate}\label{}
Much of the work needed to show that the Monge--Amp\`ere solution cannot be the Bellman function, in Case $(3_1),$ has already been started in Section \ref{subsec1_1nosolution}.  Let us finish the argument.

\begin{lemma}
In Case $(3_1),$ the solution $M$ from Lemma \ref{solution1_1and3_1} can be rewritten in the  implicit form $G(\ytwo - \yone, -\yone - \ytwo) = \ythr G(1, -\sbeta),$ where $G(z_1, z_2) = (z_1+z_2)^{p-1}[z_1-(p-1)z_2].$
\end{lemma}

\begin{proof}
Recall that in Case $(3_2)$ we have that $\ytwo < 0.$ 
\beno
\ytwo = \frac 12 (x_2 -x_1)<0 \quad\implies\quad x_2 < x_1
\eeno
\beno
\sign (\sbeta -1) = \sign \ytwo < 0 \quad\implies\quad \sbeta < 1 \quad\implies\quad \om < \sqrt{\tau^2 + 1}
\eeno
So, $\B(x) = M(y) < \ythr (\tau^2 +1)^{\frac p2}.$  Now (\ref{halfimplicitsolution1_1and3_1}) can be rewritten as
\beno
(x_1 -x_2)^{p-1}[x_1 + (p-1)x_2]=\left[x_3^{\rp} - \sqrt{\B^{\frac 2p}-\tau^2 x_3^{\frac 2p}} \right]\left[(p-1)\sqrt{\B^{\frac 2p}-\tau^2 x_3^{\frac 2p}} + x_3^{\rp}\right].
\eeno
Therefore, 
\beno
G(x_1, -x_2) = G\left(x_3^{\rp}, -\sqrt{\B^{\frac 2p}-\tau^2 x_3^{\frac 2p}}\right),
\eeno
or by factoring out $x_3^{\rp}$ on the right side we get
\beno
G(\yone - \ytwo, -\yone -\ytwo) = \ythr G(1, -\sbeta). \qedhere
\eeno
\end{proof}

Since $\ytwo$ is fixed then $D_2 \geq 0$ must be true in order that the Monge--Amp\`ere solution from Case $(3_1)$ is the Bellman function (see Proposition \ref{D_i}).  However, the contrary is true: $D_2 <0.$

\begin{lemma}\label{H''3_1}
In Case $(3_1)$ we choose $H(\yone, \ytwo) = G(y_1-y_2,-y_1+y_2)$ because of how the implicit solution is defined and obtain $\sign H'' = \sign (p-2)$
\end{lemma} 

\begin{proof}

We already computed 
   \begin{displaymath}
   H'' = \left\{
     \begin{array}{lr}
       4G_{z_1 z_2}, \al_j = \beta_j\\
       0, \al_j = -\beta_j
     \end{array}
   \right.
\end{displaymath} 

in Lemma \ref{H''1_2}.

Since, $\alpha_1 = 1, \alpha_2 = -1, \beta_1 = -1$ and $\beta_2 = 1$ then $G_{\!z_{_1}\!z_{_2}}=p(p-1)(p-2)(y_1-y_2)(2y_2)^{p-3}.$ \qedhere
\end{proof}

\begin{lemma}\label{D23_1}
If $p \neq 2$ then $D_2 < 0$ in Case $(3_1)$ for all $\tau.$
\end{lemma}

\begin{proof}
We use the partial derivatives of $G$ computed in the proof of Lemma \ref{H''1_2} to make the following computations of $\Phi'$ and $\Phi''$ easier.  Let $\beta = \sqrt{\om^2 - \tau^2}.$
\beno
\Phi(\om) & = & G(1,-\beta)\notag \\ 
\Phi'(\om) & = & -p(p-1)\om(1-\beta)^{p-2}\\ 
\Phi''(\om) & = & -p(p-1)[(1-\beta)^{p-2}-(p-2)\om^2\beta^{-1}]\\ 
\Lambda & = & (p-1)\Phi' - \om \Phi''\\
& = & p(p-1)\om(1-\beta)^{p-3}[-(p-1)(1-\beta)+(1-\beta)-(p-2)\om^2\beta^{-1}]\\ 
& = & -p(p-1)\om(1-\beta)^{p-3}(p-2)[1-\beta+\om^2\beta^{-1}] \\ 
& = & -p(p-1)(p-2)\om(1-\beta)^{p-3}\left(1+\frac{\tau^2}{\beta}\right)\\
\eeno
From Lemma \ref{sgnytwo}, $1-\beta>0$ and $\om^2 > \tau^2 > 0.$  Therefore, by Lemma \ref{H''3_1} and (\ref{Di}) $\sign D_2 = \sign H'' \sign \Lambda = - (\sign(p-2))^2 < 0.$ \qedhere
\end{proof}

Having shown that $D_2 <0$ in Case $(3_1)$ implies that the Monge--Amp\`ere solution in that case cannot be the Bellman function.

\begin{prop}
Case $(3_1)$ does not give a Bellman function candidate.  
\end{prop}

\subsection{\bf Case$(2_1)$ gives a partial Bellman function candidate}\label{}
Case $(2)$ was considered without having to fix either $\yone$ or $\ytwo$ first, so there is nothing new to do here.  Refer to Sections \ref{partial2_2subsec} and \ref{gluingp>2} for more details.

\subsection{\bf Case $(4)$ may or may not yield a Bellman function candidate}\label{}

\begin{figure}[h]
  \centering
 \setlength{\unitlength}{0.8cm}
\subfigure[$\ytwo \geq 0$]{
\begin{picture}(6,6)(0,-1)
\put(-1,0){\vector(1,0){5}} %axis
\put(0,-1.5){\vector(0,1){6}} %axis
\put(4.2,-0.1){$y_1$} %label
%\put(2.1,-.45){$k y_2$} %label
%\put(2,-.1){\line(0,1){.2}} %tick
\put(2.2,.42){$U$} %label
\put(1, -.65){$y_2$} %label
\put(0.2,4){$y_3$} %label

\qbezier(1,0)(2.2,.5)(3.4,3) %parabola

\qbezier(2,.8)(2.22,1.1)(3.25,2.7) %characteristic
%\put(2,.9){\line(0,1){2.1}} %characteristic
\put(1,-.1){\line(0,1){3.1}} %last vertical line
\end{picture}}                
\subfigure[$\ytwo \leq 0$]
{\setlength{\unitlength}{0.8cm}
\begin{picture}(6,6)(0,-1)
\put(-1,0){\vector(1,0){6}} %axis
\put(0,-1.5){\vector(0,1){6}} %axis
\put(5.2,-0.1){$y_1$} %label
\put(2.2,1.42){$U$} %label
%\put(2.1,-.45){$- k y_2$} %label
%\put(2,-.1){\line(0,1){.2}} %tick
\put(0.8, -.65){$-y_2$} %label
\put(1,-.1){\line(0,1){.2}} %tick
\put(-1.2,4){$y_3$} %label

\qbezier(1,1)(2.2,1.5)(3.4,4) %parabola
\qbezier(2,1.8)(2.22,2.1)(3.25,3.7) %characteristic
%\put(2,1.9){\line(0,1){2.1}} %characteristic
\put(1,1){\line(0,1){3.1}} %last vertical line

\end{picture}}
\caption{Characteristic of solution in Case $(4_1).$}
\label{characteristic4_1}
\end{figure}
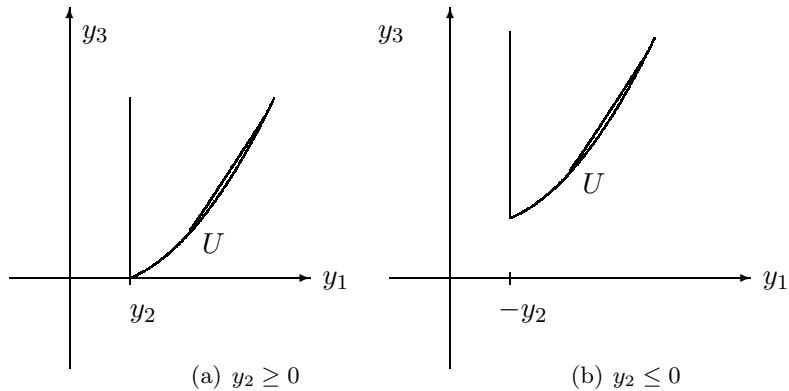
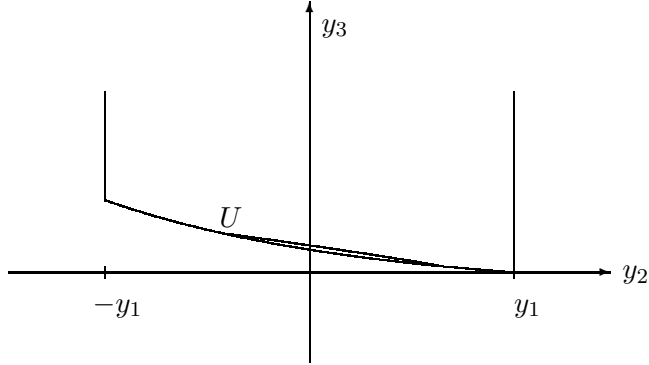
\begin{figure}[h]

\setlength{\unitlength}{0.8cm}

\begin{picture}(10,6)(-5,-1)
\put(-5,0){\vector(1,0){10}} %axis
\put(0,-1.5){\vector(0,1){6}} %axis
\put(5.2,-0.1){$y_2$} %label

\put(-3.6,-.65){$-y_1$} %label
\put(-3.4,-.1){\line(0,1){.2}} %tick

\put(-1.5,.75){$U$} %label

\put(3.4, -.65){$y_1$} %label
\put(0.2,4){$y_3$} %label

\qbezier(3.4,0)(-0.8853,0.3)(-3.4,1.2) %parabola

\put(-3.4,1.2){\line(0,1){1.8}} %first vertical line

\qbezier(-1.3,.63)(1.1,.3)(2.2,.1) %characteristic 
%\put(-1,.6){\line(0,1){2.4}}
\put(3.4,-.1){\line(0,1){3.1}} %last vertical line

\end{picture}
\caption{Characteristic for the solution from Case $(4_2)$} 
\label{characteritic4_2}
\end{figure}

For $\tau = 0,$ it was shown in \cite{VaVo2} that Case $(4)$ does not produce a Bellman function candidate, since some simple extremal functions give a contradiction to linearity of the Monge--Amp\`ere solution on characteristics.  However, for $\tau \neq 0$ it is much more difficult to show this.  Those same extremal functions do contradict linearity for some $p-$values and some signs of the Martingale transform.  For the sign of the Martingale transform where we do not have a contradiction, a new set of test of extremal functions would have to be found.  Since the Bellman function has already been constructed from other cases, this case has not been investigated any further than just described.  So, for $p$ and $\tau$ values not mentioned in the main result, Case $(4)$ could give a Bellman candidate throughout $\Xi_+$ or we could get a partial Bellman candidate that may work well with the characteristics from Case $(2_1).$

\bibliographystyle{amsplain}

\end{document}